\documentclass{article}
\usepackage[utf8]{inputenc}
\usepackage{amsmath}
\usepackage{algorithm}
\usepackage{algpseudocode}
\usepackage{amsthm}
\usepackage{amsfonts}
\usepackage{hyperref}
\usepackage{xcolor}
\hypersetup{
    colorlinks,
    linkcolor={red!80!black},
    citecolor={blue!80!black},
    urlcolor={blue!80!black}
}
\usepackage{float}
\usepackage[margin=1in]{geometry}
\usepackage{graphicx}
\usepackage{subcaption}
\usepackage{caption}
\usepackage{url}
\usepackage{bm}
\usepackage{enumitem}

\newtheorem{remark}{Remark}

\newtheorem{theorem}{Theorem}
\newtheorem{proposition}{Proposition}
\newtheorem{definition}{Definition}

\newcommand{\Real}{\mathbb R}

\newcommand{\qt}{\tilde{q}}
\newcommand{\ubar}{\overline{u}}
\newcommand{\ubarinf}{\sigma}

\newcommand{\us}{u^{*}}
\newcommand{\qzero}{q_{0}}
\newcommand{\qone}{q_{1}}
\newcommand{\dqzero}{\partial_x q_{0}}
\newcommand{\dqone}{\partial_x q_{1}}
\newcommand{\qzeros}{q_{0}^{*}}
\newcommand{\qones}{q_{1}^{*}}
\newcommand{\dqzeros}{\partial_x q_{0}^{*}}

\newcommand{\Hmod}{\bar{H}}
\newcommand{\Ionemod}{\bar{I}_1}
\newcommand{\Itwomod}{\bar{I}_2}

\newcommand{\ddqzero}{\partial_{xx} q_{0}}
\newcommand{\qzzv}{(\bm{\qzero})_0}
\newcommand{\qozv}{(\bm{\qone})_0}

\newcommand{\qbp}{\qt}
\newcommand{\qbm}{\qt}

\DeclareMathOperator{\sech}{sech}

\title{A Hyperbolic Approximation of the Nonlinear Schr\"odinger Equation}

\author{Abhijit Biswas\thanks{Corresponding author.} \and Laila S. Busaleh \and David I. Ketcheson \and Carlos Mu\~noz-Moncayo \and Manvendra Rajvanshi}

\begin{document}
\maketitle

\abstract{
    We study a first-order hyperbolic approximation of
    the nonlinear Schr\"odinger (NLS) equation.  We show that the system
    is strictly hyperbolic and possesses a modified Hamiltonian structure, along with
    at least three conserved quantities that approximate those of NLS. 
    We provide families of explicit standing-wave solutions to the hyperbolic system,
    which are shown to converge uniformly to ground-state solutions
    of NLS in the relaxation limit.
    The system is formally equivalent to NLS in the relaxation limit, and we
    develop asymptotic preserving discretizations that tend to a consistent discretization
    of NLS in that limit, while also conserving mass.
    Examples for both the focusing and defocusing regimes demonstrate that the
    numerical discretization provides an accurate approximation of the NLS
    solution.
}

\section{Introduction}
We consider the nonlinear Schr\"{o}dinger (NLS) equation
\begin{align}\label{Eq:NLS}
i u_t + u_{xx} + \kappa |u|^2 u & = 0, \quad -\infty < x < \infty, \quad t > 0,
\end{align}
where $u(x,t)$ is a complex-valued function and $\kappa \in \Real$. 
The NLS equation arises in a variety of physical applications, including
fiber optics, Bose-Einstein condensates, and deep water waves.
It possesses a rich mathematical structure and there is an extensive
literature regarding its solution by numerical or analytical means;
see e.g. \cite{yang2010nonlinear} and references therein.

Recently, first-order hyperbolic approximations have been proposed for a variety of
higher-order partial differential equations (PDEs), especially dispersive nonlinear
wave equations (see for example \cite{toro2014advection,mazaheri2016first,chesnokov2019hyperbolic,gavrilyuk2022hyperbolic}).
Among these is a hyperbolic approximation of the (NLS) equation,
proposed by Dhaouadi et. al. \cite{dhaouadi2019extended}.
Unlike other dispersive wave equations, which have third- or higher-order
derivatives, the NLS equation is second-order and complex-valued.
The approach in \cite{dhaouadi2019extended} is based on 
the Madelung transformation, which recasts the NLS equation as a more
typical third-order, real-valued nonlinear dispersive wave system.
The authors develop a first-order hyperbolic approximation of the latter
system of equations.  Therein it is suggested that the
hyperbolic approximation may facilitate the enforcement of relevant conservation
laws and the implementation of non-reflecting boundary conditions.

Here we consider a more direct approach, by a hyperbolic relaxation of the NLS
equation in its natural form \eqref{Eq:NLS}.  Specifically, we consider the 
hyperbolized NLS (referred to herein as NLSH) system 
\begin{subequations} \label{nlsH}
\begin{align}
    i\partial_t q_0 + \partial_x q_1 &= - \kappa |q_0|^2 q_0 \;, \label{nlsH1} \\
    i \tau \partial_t q_1 &= \partial_x q_0 - q_1 \;,
\end{align}
\end{subequations}
which is strictly hyperbolic for all values of the relaxation parameter $\tau>0$.
Formally, as $\tau \to 0$, we have $q_1 \to \partial_x q_0$ and then \eqref{nlsH1}
becomes equivalent to \eqref{Eq:NLS} with $q_0 \approx u$.  The NLSH system was proposed in
\cite{ketcheson2025approximation} and is studied in depth here for the first
time.
Compared to the hyperbolic approximation given in \cite{dhaouadi2019extended},
NLSH avoids the need for the Madelung transformation, yields a
simpler system of equations, and preserves more of the mathematical structure of
the NLS equation.  Specifically, NLSH possesses at least three exactly conserved
quantities that approximate three invariants of NLS and are formally equivalent in the relaxation
limit.  Furthermore, whereas the previous approach of \cite{dhaouadi2019extended} can
be used only in the defocusing ($\kappa<0$) case, the NLSH system is valid for approximation 
of both focusing and defocusing NLS.  
Finally, the system developed here has a hyperbolic part that is linear and
is strictly (symmetric) hyperbolic for all values of the relaxation parameter.

The rest of the paper is organized as follows.  In Section \ref{sec:structure} we
review the Hamiltonian structure of NLS and then analyze the corresponding structure
of NLSH, proving that NLSH is a hyperbolic system with three nonlinear conserved quantities.
In Section \ref{sec:standing-waves} we study standing-wave solutions of
the NLSH system, providing explicit solutions and proving that they converge to
ground-state solutions of NLS 
uniformly and linearly in $\tau$.  
In Section \ref{sec:discretization} we propose an implicit-explicit
numerical discretization of NLSH and show that it is mass conserving, asymptotic preserving and,
for some time discretizations, asymptotically accurate.
In Section \ref{sec:tests} we illustrate the behavior of the numerical methods and
the NLSH system itself through numerical solutions.

\section{Hamiltonian structure} \label{sec:structure}
In this section we briefly review the Hamiltonian structure and
conserved quantites of the NLS equation
and then study the analogous structures and conserved quantities of the NLSH system.

\subsection{NLS as a Hamiltonian system}
Let $\us$ denote the conjugate field of the complex field $u$. The NLS equation \eqref{Eq:NLS} for $u$ can be 
written as a Hamiltonian partial differential equation (PDE) \cite[Eqn.~64.9]{lanczos1949}:
\begin{align}
i\frac{\partial u}{\partial t} & = \frac{\delta H}{\delta \us} \;,
\end{align}
where
\begin{align}\label{Eq:Hamilton_NLS}
H = \int_{\mathbb{R}} \left(|u_x|^2 - \frac{\kappa}{2}|u|^4\right) dx,
\end{align}
and $\frac{\delta H}{\delta \us}$ is the variational derivative of the Hamiltonian with respect to $\us$. 
It is well known that the NLS equation possesses infinitely many conserved quantities, of which the first two are given by
\cite{yang2010nonlinear}
\begin{subequations} \label{NLS-conserved}
\begin{align}
    I_{1}(t) &= - \int_{\mathbb{R}} |u|^2 \, dx \;,
\end{align}
and
\begin{align}
    I_{2}(t) &= -i \int_{\mathbb{R}} \us u_x \, dx \;.
\end{align}
\end{subequations}
In accordance with Noether's Theorem, these can be associated with symmetries of NLS through the flows
\begin{align}
    \begin{aligned}
        i\frac{\partial u}{\partial t} &= \frac{\delta I_1}{\delta \us} = -u\;,
    \end{aligned}
    & \quad
    \begin{aligned}
        i\frac{\partial u}{\partial t} &= \frac{\delta I_2}{\delta \us} = -i u_x \;,
    \end{aligned}
\end{align}
which represent rotation and translation of the solution $u$, respectively. 
Note that the factors $-1$ and $-i$ in \eqref{NLSH-conserved} are included to ensure that these 
flow equations are consistent with those in \cite{duran2000numerical}.
The system \eqref{Eq:NLS} possesses solitary wave 
solutions whose time evolution involves only rotation (in the complex plane) and/or translation
along the $x$-axis. 

\subsection{NLSH as a Hamiltonian system}
The NLSH system \eqref{nlsH} \cite{ketcheson2025approximation} has a relative equilibrium structure similar to that of 
the NLS equation. Consider the modified Hamiltonian function $\bar{H}$
\begin{align}
    \Hmod = \int_{\mathbb{R}} \left( \qones \dqzero  + \qone \dqzeros  - |\qone|^2 - \frac{\kappa}{2}|\qzero|^4\right) dx \;,
    \label{Eq:Hamilton_NLSH_modified}
\end{align} 
where $\qzeros$ and $\qones$ denote the conjugate fields of the complex fields $\qzero$ and $\qone$, respectively.
The NLSH system \eqref{nlsH} can be written in Hamiltonian PDE form as $i Q_t = \bar{\mathcal{J}} \frac{\delta \bar{H}}{\delta Q^*}$,
with
\begin{align*}
    Q =  \begin{bmatrix}
         \qzero \\
         \qone 
     \end{bmatrix}, \quad
     \bar{\mathcal{J}} = 
     \begin{bmatrix}
        1 & 0 \\
        0 & \tau^{-1}  
     \end{bmatrix}, \quad
     \frac{\delta \bar{H}}{\delta Q^*} = 
     \begin{bmatrix}
        \frac{\delta \bar{H}}{\delta \qzeros} \\
        \frac{\delta \bar{H}}{\delta \qones}   
     \end{bmatrix} \;,
\end{align*}
where $\frac{\delta \bar{H}}{\delta \qzeros}$ and $\frac{\delta \bar{H}}{\delta \qones}$ are the variational derivatives of 
the modified Hamiltonian $\Hmod$ with respect to $\qzeros$ and $\qones$, respectively, and are given by
\begin{align}
     \begin{bmatrix}
        \frac{\delta \bar{H}}{\delta \qzeros} \\
        \frac{\delta \bar{H}}{\delta \qones} 
     \end{bmatrix} & = 
     \begin{bmatrix}
         - \dqone - \kappa |\qzero|^2\qzero \\
         \dqzero - \qone  
     \end{bmatrix}
     \;.
\end{align}
Besides the Hamiltonian, the NLSH system also implies conservation of the quantities
\begin{subequations} \label{NLSH-conserved}
\begin{align}
    \Ionemod(t) &=  - \int_{\mathbb{R}} \left( |\qzero|^2 + \tau |\qone|^2\right) dx \;, \\
    \Itwomod(t) &= -i \int_{\mathbb{R}} \left( \qzeros  \dqzero + \tau \qones  \dqone \right) dx, \;
\end{align}
\end{subequations}
which are modifications of the first integrals $I_1$ and $I_2$, respectively, of the original NLS equation \eqref{Eq:NLS}.
It can be verified by direct computation that
\begin{subequations} \label{NLSH-conservation-laws}
\begin{align}
    \int_{\mathbb{R}} \frac{\partial}{\partial t}\left[- \left(|\qzero|^2 + \tau |\qone|^2 \right) \right] dx = -i \int_{\mathbb{R}} \frac{\partial}{\partial x} \left[ \qzeros \qone - \qzero \qones \right] dx \;,
\end{align}
and 
\begin{align}
    \int_{\mathbb{R}} \frac{\partial}{\partial t} \left[-i \left(\qzeros  \dqzero + \tau \qones  \dqone \right) \right] dx = - i \int_{\mathbb{R}} \frac{\partial}{\partial x} \left[ i \qzeros  \dqone - i \qones \dqzero + \frac{i \kappa}{2}|\qzero|^4 \right] dx \;.
\end{align}
\end{subequations}
so $\bar{I}_1$ and $\bar{I}_2$ are conserved for the hyperbolized system \eqref{nlsH}. Note that the NLSH 
system \eqref{nlsH} can be written as $A Q_t + B Q_x = S(Q)$, with 
\begin{align*}
    A = \begin{bmatrix} 
        1  & 0 \\
        0 & \tau 
    \end{bmatrix} \;, \quad
    B = 
    \begin{bmatrix}
        0  & -i \\
        i & 0 
    \end{bmatrix} \;.
\end{align*}
Here $A$ and $B$ are Hermitian, and $A^{-1}B$ has eigenvalues $\pm
\sqrt{\tau^{-1}}$ that are distinct and real for all $\tau>0$.  Thus we have
proved the following: 
\begin{theorem}
    The NLSH equations \eqref{nlsH} with $\tau>0$ are strictly hyperbolic and possess
    three conserved quantities: $\Hmod, \Ionemod, \Itwomod$.
\end{theorem}


\begin{remark}
   The flows of the vector fields given by $\bar{\mathcal{J}} \delta \bar{I}_1$ and $\bar{\mathcal{J}} \delta \bar{I}_2$ 
   correspond respectively to rotation and translation of the solution.
\end{remark}

\begin{remark}
   The wave speeds $\sqrt{\tau^{-1}}$ become arbitrarily large as $\tau\to0$,
   leading to stiffness and necessitating the use of implicit time integration;
   see Section \ref{sec:imex}.
\end{remark}

\section{Standing-wave solutions} \label{sec:standing-waves}
A key feature of the NLS equation is the existence of ground state solutions \cite{yang2010nonlinear},
which have the form
\begin{align}
\label{eq:NLS ground state}
u(x,t) = \exp \left({i\mu t }\right) u^\pm(x).
\end{align}
Here, $\mu \in \Real$ and $u^\pm$ is a real-valued function given by
\begin{subequations} \label{standing-amplitude}
\begin{align}
    u^+(x) & =\sqrt{2}\ubarinf \mathrm{sech}(\sqrt{\mu}x \pm K_1), \quad  \text{if } \; \mu, \kappa > 0, \label{eq:explicit_solitary_NLS}\\
    u^-(x) & =\ubarinf \mathrm{tanh}(\pm \ubarinf \sqrt{-{\kappa}/{2}}x\pm K_0 ), \quad  \text{if } \; \mu, \kappa <0, \label{eq:Exact_traveling_front_NLS}
\end{align}
\end{subequations}
where
\begin{align}
    \label{eq:NLS bound state constants}
\ubarinf = \sqrt{\frac{\mu}{\kappa}},
\quad K_0 = \mathrm{tanh}^{-1}\left(\frac{u^-(0)}{\ubarinf}\right),
\quad K_1 = \mathrm{sech}^{-1}\left(\frac{u^+(0)}{\sqrt{2}\ubarinf}\right).
\end{align}

The goal of this section is to show the existence of standing-wave solutions of the NLSH system \eqref{nlsH} 
and analyze their relation to the ground state solutions of the NLS equation.
For this purpose we consider the standing-wave ansatz
\begin{align}
    \label{eq:translation_rotation_ansatz}
    q_{g}(x,t)
    =
    \mathrm{exp}\left({i\mu t}\right) \qt(x),
\end{align}
where $\mu$ is a real parameter and $\qt(x) = [\qt_0(x), \qt_1(x)]^T$ is a vector-valued function.
We obtain then the following system of ordinary differential equations (ODEs):
\begin{subequations}
    \label{eq:standing_wave_NLSH}
\begin{align}
\qt_0' &= (1-\mu\tau)\qt_1,\label{Eq:travel_wave_q0} \\
\qt_1' &= \left({\mu}-\kappa {\qt_0}^2\right)\qt_0  \label{Eq:travel_wave_q1}.
\end{align}
\end{subequations}

System \eqref{eq:standing_wave_NLSH} has three equilibrium points $(\qt_0,\qt_1)$:
\begin{align}
\left( - \ubarinf, 0 \right), \quad (0, 0), \quad \left( + \ubarinf, 0 \right).
\end{align}
We are mainly interested in solitary waves and fronts, so we assume from now
on that $\kappa$ and $\mu$ have the same sign, in which case all three equilibria
are real (otherwise, there is only one real equilibrium and no localized waves or
fronts are possible).
The Jacobian of the right-hand side of \eqref{eq:standing_wave_NLSH} is
\begin{align}
J(\qt) = \begin{bmatrix}
         0 & 1 - \tau \mu \\
         \mu - 3 \kappa \qt_0^2 & 0 
    \end{bmatrix},
\end{align}
and has eigenvalues
\begin{align}
\lambda_{1,2} = \pm \sqrt{(\mu - 3 \kappa {\qt_0}^2)(1 - \tau \mu)}.
\end{align}


Thus the local dynamics of eq. \eqref{eq:standing_wave_NLSH} are determined by the sign of $\kappa$
(or $\mu$) and of $1-\tau\mu$.  There are three interesting cases:
\begin{itemize}
    \item $\kappa<0$ (defocusing), $\mu<0$ :  In this case, shown in Figure
        \ref{fig:phase_plane2}, the left and right equilibria are saddles, with
        heteroclinic connections between them representing fronts, which we study
        in Section \ref{sec:traveling-fronts-defocusing}.
    \item $\kappa>0$ (focusing), $\mu>0$, $\mu \tau>1$:  The topology is again as shown in
        Figure \ref{fig:phase_plane2}.
    Thus, it is also possible to find standing front solutions in this regime.
    However, these solutions arise exclusively as a consequence of the hyperbolic approximation,
    and they do not correspond to any localized solutions of the NLS equation.
    Moreover, proceeding as in Proposition \ref{prop:uniform_convergence front} below, one can show that these
    solutions converge uniformly and linearly in $\tau\mu$ to zero as $\tau\mu \to 1$.
    \item $\kappa>0$ (focusing), $\mu>0$, $\mu \tau<1$: In this case, shown in
        Figure \ref{fig:phase_plane}, the middle equilibrium is a saddle while the left
        and right points are centers.  There exist homoclinic connections, corresponding
        to solitary wave solutions, which we study in Section \ref{sec:solitary-waves-focusing}.
\end{itemize}


\begin{figure}
    \centering
    \begin{subfigure}[b]{0.49\textwidth}
        \centering
        \includegraphics[width=\textwidth]{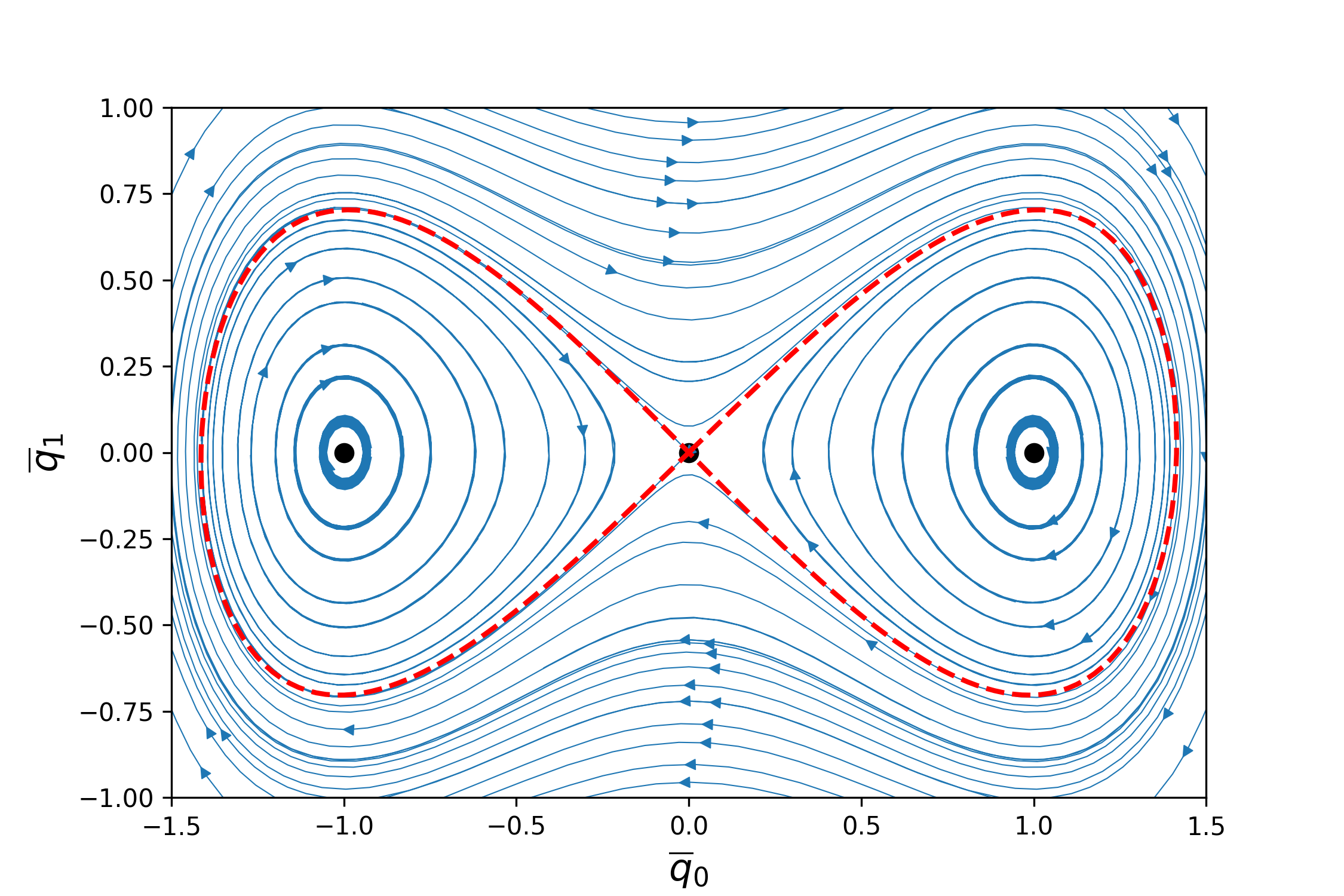}
        \caption{Focusing case, $\kappa=1$: homoclinic orbits of $\qt^*=(0,0)$
        separate the phase plane into three regions}
        \label{fig:phase_plane}
    \end{subfigure}
    \hfill
    \begin{subfigure}[b]{0.49\textwidth}
        \centering
        \includegraphics[width=\textwidth]{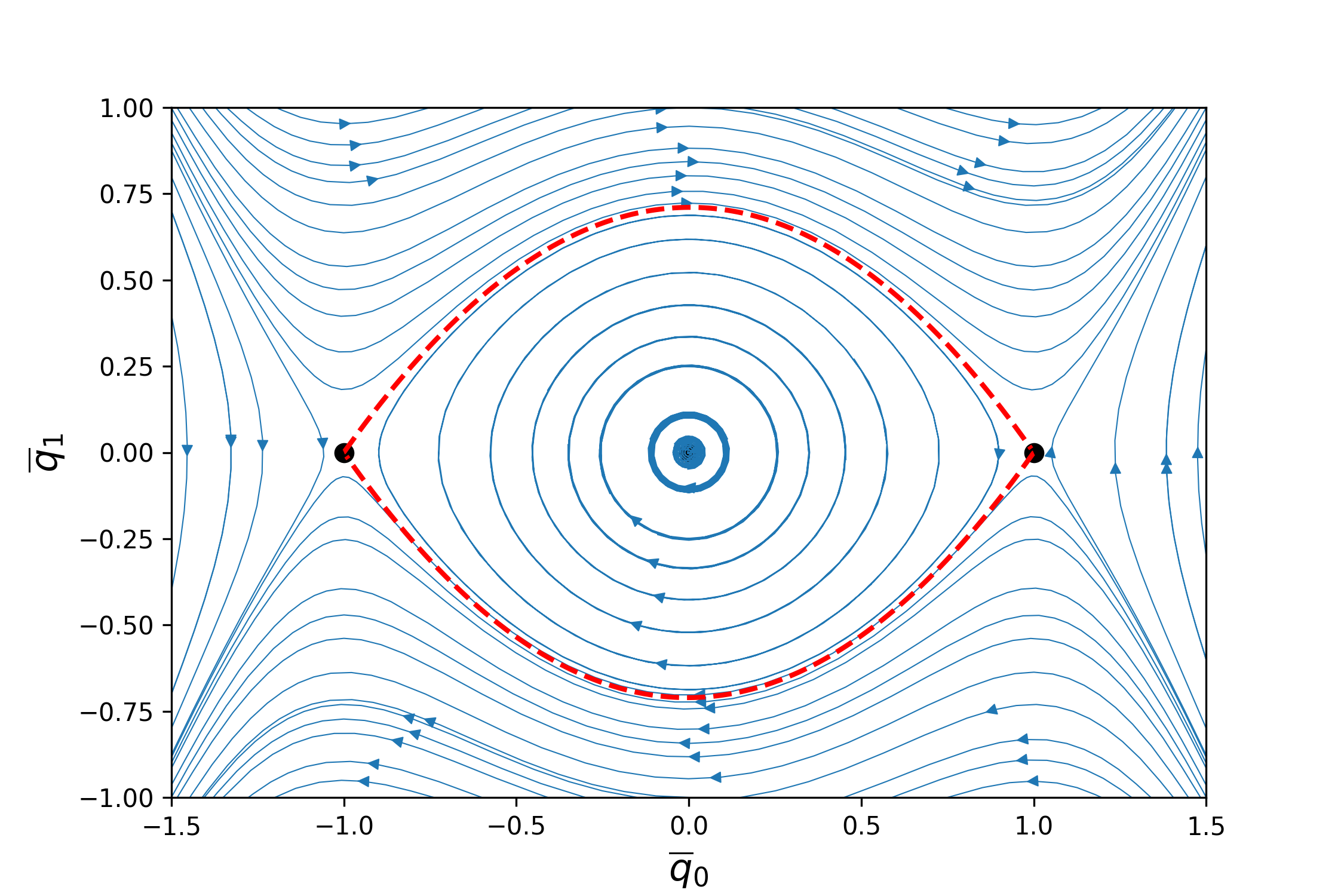}
        \caption{Defocusing case, $\kappa=-1$: heteroclinic orbits connect the equilibria
        $\qt^{\pm}=(\pm \ubarinf,0)$
        }
        \label{fig:phase_plane2}
    \end{subfigure}
    \caption{Phase portraits illustrating the ground state solutions of the NLSH system with $\tau=10^{-3}$ and
    $\mu=\kappa$}
    \label{fig:phase_portraits}
\end{figure}

\subsection{Standing fronts}
\label{sec:traveling-fronts-defocusing}
In the defocusing case ($\kappa, \mu <0$) it is possible to find explicit stationary front solutions for the NLSH system \eqref{eq:standing_wave_NLSH}.
Moreover, we can show that they converge to the amplitude \eqref{eq:Exact_traveling_front_NLS} of 
\emph{dark} soliton solutions to the NLS equation \eqref{Eq:NLS} in the limit $\tau \to 0$.

By substituting \eqref{Eq:travel_wave_q0}
into \eqref{Eq:travel_wave_q1}, we get
\begin{align}
    \label{eq:ODE_NLSH}
    \frac{1}{1-\mu \tau} \qbm_0'' - \mu \qbm_0 + \kappa \qbm_0^3 = 0,
\end{align}
which, after multiplying by $\qbm_0'$ and integrating once, yields
\begin{align}
    \label{eq:weak_ODE_NLSH}
    \frac{1}{1-\mu\tau}[\qbm_0']^2 -  \mu \qbm_0^2 + \frac{\kappa}{2} \qbm_0^4 = \frac{\kappa}{2}\ubarinf^4-\mu\ubarinf^2.
\end{align}
Here, we have assumed that $\lim_{x\to +\infty} |\qbm_0(x)| 
= \ubarinf$ and $\lim_{x\to +\infty} \qbm_1(x) = \lim_{x\to +\infty} \qbm_0'(x) = 0$,
which are the boundary conditions we expect for a standing front solution of the NLSH system
from the phase plane analysis conducted earlier.
Using separation of variables, equation \eqref{eq:weak_ODE_NLSH} can be  written as
\begin{subequations}
    \label{eq:ODE_NLSH_separated}
\begin{align}
    \frac{1}{\sqrt{{1-\mu\tau}}}\int \frac{d\qbm_0}{\sqrt{-\frac{\kappa}{2}(\qbm_0)^4
    +\mu(\qbm_0)^2+\frac{\kappa}{2}\ubarinf^4-\mu\ubarinf^2}} &=  \pm \int dx \;, \\
    \frac{1}{\sqrt{{1-\mu\tau}}} \int \frac{d\qbm_0}{\sqrt{
        (\ubarinf^2-(\qbm_0)^2)[-\mu + \frac{\kappa}{2}(\ubarinf^2+(\qbm_0)^2)] }} &=  \pm \int dx \;, \\  
    \sqrt{-\frac{2}{\kappa({1-\mu\tau})}}\int \frac{d\qbm_0}{\ubarinf^2-(\qbm_0)^2} &=  \pm \int dx \;.
\end{align}
\end{subequations}


From \eqref{eq:ODE_NLSH_separated}, and assuming that  $\qbm_0(0) = u^-(0)$, we obtain the closed-form expressions
\begin{subequations} \label{eq:Exact_traveling_front_NLSH}
\begin{align}
    \qbm_0^-(x) := \ubarinf \mathrm{tanh}\left(\pm \ubarinf \sqrt{\frac{\kappa(\mu\tau-1)}{2}}x\pm K_0
    \right),
\end{align}
and 
\begin{align}
    \qbm_1^-(x) := \frac{1}{1-\mu\tau}\qbm_0'(x) = \frac{\pm \ubarinf^2 \sqrt{\frac{\kappa(\mu\tau-1)}{2}}}{1-\mu\tau}
    \mathrm{sech}^2\left(\pm \ubarinf \sqrt{\frac{\kappa(\mu\tau-1)}{2}}x\pm K_0\right).
\end{align}
\end{subequations}

\begin{proposition}
    \label{prop:uniform_convergence front}
    The family of standing front solutions $\{\qbm_0\}_{\tau>0}$ \eqref{eq:Exact_traveling_front_NLSH} to the defocusing NLSH system \eqref{eq:ODE_NLSH} converges uniformly, and linearly in $\tau$, to the 
    ground state solution $\ubar^{-}$ \eqref{eq:Exact_traveling_front_NLS} of the defocusing NLS equation as $\tau\to 0$.
    Moreover, $\{\qbm_1\}_{\tau>0}$ also converges uniformly and at a linear rate to $(\ubar^{-})'$.
\end{proposition}
\begin{proof}
    Let $x\in \Real\setminus \{0\}$ ($\qbm_0(0)= u^{-}(0)$) and $\tau>0$. By applying the Mean Value Theorem to the function
    \begin{align}
        f(x) = \tanh\left(\pm \ubarinf \sqrt{\frac{-\kappa}{2}}x \pm K_0\right),
    \end{align}
    we have that there exists some $\xi$ in the interval
    \begin{align}
        \label{eq:MVT interval}
        \left[\sqrt{1-\mu\tau} x,x\right] \quad \text{or} \quad \left[x,\sqrt{1-\mu\tau} x\right]
    \end{align}
    (depending on the sign of $x$) such that
    \begin{subequations}
        \label{eq:applying MVT}
        \begin{align}
        |\qbm_0(x) - u^{-}(x)| &= |\ubarinf| \left| \tanh\left(\pm \ubarinf \sqrt{\frac{\kappa(\mu\tau-1)}{2}}x \pm K_0\right) - \tanh\left(\pm \ubarinf \sqrt{-\kappa/2} x \pm K_0\right) \right| \\
        &\leq \ubarinf^2 \sqrt{-\kappa/2} \left|\sqrt{1-\mu\tau}-1\right| |x|\mathrm{sech}^2\left(\pm \ubarinf \sqrt{\frac{-\kappa}{2}}\xi \pm K_0\right).
    \end{align}
    \end{subequations}
    From \eqref{eq:MVT interval}, we have that 
    \begin{align}
        \label{eq:set MVT}
        \xi \in S(x,\tau) :=\left[-\sqrt{1-\mu\tau}|x|,-|x|\right] \cup \left[|x|,\sqrt{1-\mu\tau}|x|\right].
    \end{align}
    Without loss of generality, we can assume that $\tau$ is bounded from above.
    Then, there is some $\tau_0>0$ such that 
    \begin{align}
        \label{eq:set nesting}
         S(x,\tau) \subseteq S(x,\tau_0),\quad  \forall \tau\in]0,\tau_0[\,.
    \end{align}
    From \eqref{eq:applying MVT}, \eqref{eq:set MVT}, and \eqref{eq:set nesting}, we obtain
    \begin{align}
        \label{eq:uniform_conv_1}
        |\qbm_0(x) - u^{-}(x)| \leq \ubarinf^2 \sqrt{-\kappa/2} |\sqrt{1-\mu\tau}-1| |x|\mathrm{sech}^2\left(\pm \ubarinf \sqrt{\frac{-\kappa}{2}}\overline{\xi}|x| \pm K_0\right),
    \end{align}
    for some $\overline{\xi}\in \overline{S}(\tau_0):=[-\sqrt{1-\mu\tau_0},-1] \cup [1,\sqrt{1-\mu\tau_0}]$.
    Due to the exponential decay of the $\mathrm{sech}^2$ function and since $\overline{S}(\tau_0)$ does not depend on $x$,
    we can find a constant $L_0<+\infty$ such that
    \begin{align}
        \label{eq:bound_sech}
        |x| \mathrm{sech}^2\left(\pm \ubarinf \sqrt{-\frac{\kappa}{2}} \overline{\xi} |x| \pm K_0\right) <L_0, \quad \forall x\in \Real.
    \end{align}
    
    Taking a Taylor expansion of $\sqrt{1-\mu \tau}$ around $\tau=0$, we have that $\sqrt{1-\mu \tau} = 1 - \frac{\mu \tau}{2} + \mathcal{O}(\mu^2 \tau^2)$.
   Thus, from \eqref{eq:uniform_conv_1} and \eqref{eq:bound_sech}, we obtain
    \begin{align}
        \label{eq:uniform_conv_2}
         \sup_{x\in\mathbb{R}}|\qbm_0(x) - u^{-}(x)| \leq L_0 \ubarinf^2 \sqrt{-\kappa/2} \left(|\mu\tau/2| + |\mathcal{O}(\mu^2 \tau^2)|\right).
    \end{align} 
    One can follow a similar argument to show the uniform convergence of $\qbm_1$ to $\ubar'$.
\end{proof}

The standing front solutions \eqref{eq:Exact_traveling_front_NLSH} to the NLSH system 
(for different values of $\tau$)
and corresponding solutions \eqref{eq:Exact_traveling_front_NLS} of the NLS
equation are depicted in Figure \ref{fig:traveling_fronts}.

\begin{figure}
    \centering
    \includegraphics[width=0.85\textwidth]{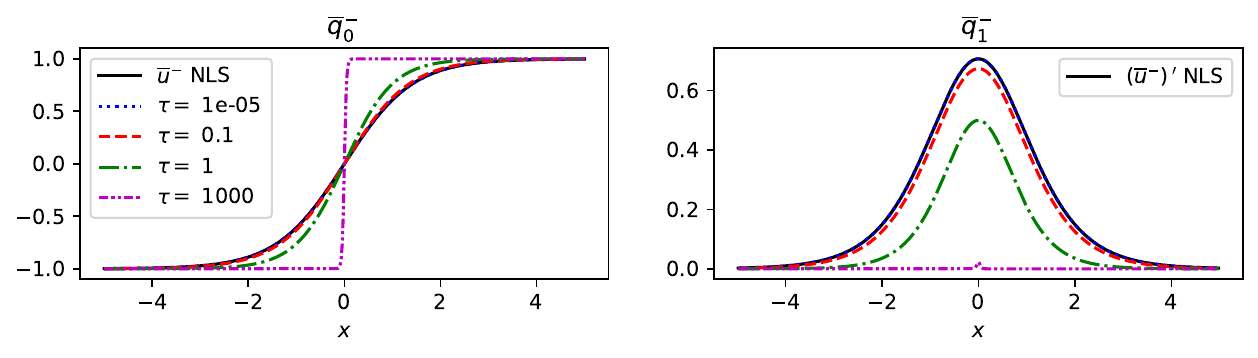}
    \caption{Standing front solutions to the  NLSH
    system \eqref{eq:Exact_traveling_front_NLSH}  and the NLS equation \eqref{eq:Exact_traveling_front_NLS}
    with $\kappa=\mu=-1$, $K_0=0$, and different values of $\tau$}
    \label{fig:traveling_fronts}
\end{figure}

\subsection{Standing solitary waves}
\label{sec:solitary-waves-focusing}
In the focusing case ($\kappa, \mu >0$ and $\tau \mu<1$)
it is also
possible to find explicit solitary wave solutions for the NLSH system \eqref{eq:standing_wave_NLSH}
that converge to the amplitude \eqref{eq:explicit_solitary_NLS} of
\emph{bright} soliton solutions to the NLS equation \eqref{Eq:NLS} as $\tau\to 0$.
Again, we multiply \eqref{eq:ODE_NLSH} by $\qt_0'$ and integrate once, but now with the boundary condition
$\lim_{x\to +\infty} \qt_0(x) = 0$ as expected from the phase plane analysis.
The resulting ODE can be solved explicitly, leading to
\begin{subequations} \label{eq:explicit_solitary_NLSH}
\begin{align}
    \label{eq:explicit_solitary_NLSH_q0}
    \qbp_0^+(x) := \sqrt{2}\ubarinf\mathrm{sech}(\sqrt{\mu(1-\mu\tau)}x \pm K_1),
\end{align}
and
\begin{align}
    \label{eq:explicit_solitary_NLSH_q1}
    \qbp_1^+(x) := \frac{- \sqrt{2 \mu }\ubarinf}{\sqrt{(1-\mu\tau)}}\mathrm{sech}(\sqrt{\mu(1-\mu\tau)}x \pm K_1)\tanh(\sqrt{\mu(1-\mu\tau)}x \pm\ K_1),
\end{align}
\end{subequations}
where we have imposed the condition $\qt_0(0) = u^{+}(0)$. 
In Figure \ref{fig:solitary_waves}, we depict the solitary wave solutions
\eqref{eq:explicit_solitary_NLSH} to the NLSH system (for different values of $\tau$)
and their counterpart solutions \eqref{eq:explicit_solitary_NLS} of the NLS equation.
Proceeding as in the previous section, the following proposition can be proven.
\begin{proposition}
    \label{prop:uniform_convergence_solitary}
    The family of solitary wave solutions $\{\qbp_0\}_{\tau>0}$ \eqref{eq:explicit_solitary_NLSH_q0} to the focusing NLSH system \eqref{eq:ODE_NLSH} converges uniformly,
    and linearly in $\tau$, to the 
    ground state solution $u^{+}$ \eqref{eq:explicit_solitary_NLS} of the focusing NLS equation as $\tau\to 0$.
    Moreover, $\{\qbp_1\}_{\tau>0}$ also converges uniformly and at a linear rate to $(u^{+})'$.
\end{proposition}
\begin{figure}
    \centering
    \includegraphics[width=0.85\textwidth]{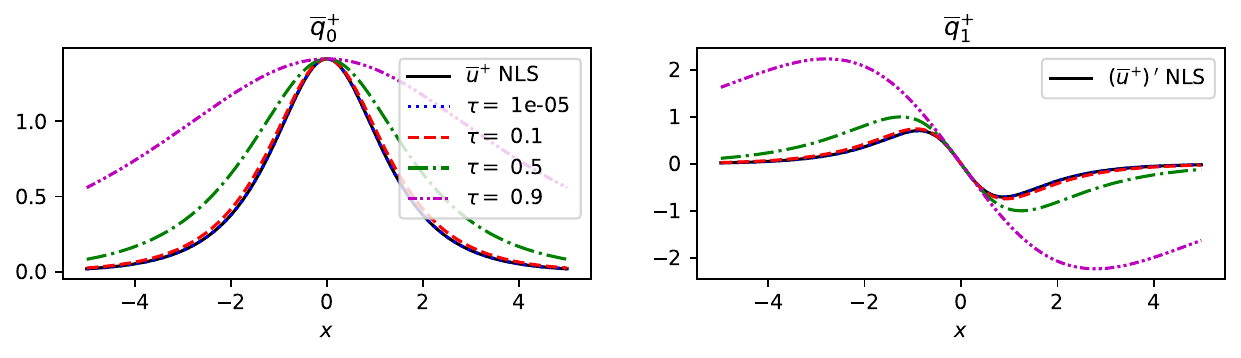}
    \caption{Solitary wave solutions to the  NLSH
     system \eqref{eq:explicit_solitary_NLSH_q0}  and the NLS equation \eqref{eq:explicit_solitary_NLS} with $\kappa=\mu=1$, $K_1=0$, and different values of $\tau$}
    \label{fig:solitary_waves}
\end{figure}

\section{Asymptotic-Preserving and Energy-Conserving Numerical Discretization} \label{sec:discretization} 
In this section, we propose a structure-preserving discretization for the NLSH
system~\eqref{nlsH}.  By structure preserving, we mean here that:
\begin{itemize}
    \item the scheme conserves a discrete analog of the mass, and
    \item the scheme is asymptotic preserving, as defined below.
\end{itemize}
Mass conservation is obtained by choosing a mass-conserving spatial
discretization (described in Section \ref{sec:space-disc}) and then using relaxation
in time (as described in Section \ref{sec:time-relaxation}).  The asymptotic preserving
property is proved for the continuous-space, discrete-time scheme,
using ImEx Runge-Kutta discretization as described in Sections \ref{sec:imex}-\ref{sec:AP}.

\subsection{Mass-Conserving Spatial Semidiscretization} \label{sec:space-disc}
We introduce a uniform spatial grid $x_1, \dots, x_j$ and discrete in space
approximation vectors $Q_0, Q_1$ with $(Q_0)_j(t) \approx q_0(x_j,t)$ and
$(Q_1)_j(t)\approx q_1(x_j,t)$.
A natural semi-discretization of \eqref{nlsH} takes the form
\begin{subequations} \label{semidisc}
\begin{align}
    \partial_t Q_0 & = i D Q_1 + i \kappa L Q_0 \\
    \tau \partial_t Q_1 & = iQ_1 - i D Q_0,
\end{align}
\end{subequations}
where $L$ is the diagonal matrix with entries $\ell_{jj}=|(Q_0)_j|^2$
and $D$ is a matrix that approximates the first derivative operator
(with appropriate boundary conditions).

\begin{theorem}
If $D$ is skew-hermitian, then solutions of \eqref{semidisc} are mass conservative:
\begin{align} \label{discrete-mass-conservation}
    \frac{d}{dt} \left( \|Q_0\|^2 + \tau\|Q_1\|^2 \right) = 0.
\end{align}
\end{theorem}
\begin{proof}
The proof is by direct calculation.  We have (noting that $L=L^*$)
\begin{align*}
    \frac{d}{dt} \|Q_0\|^2 & = Q_0^* \partial_t Q_0 + (\partial_t Q_0^*)Q_0 \\
        & = i Q_0^* D Q_1 + i \kappa Q_0^* L Q_0 - i Q_1^* D^* Q_0 - i \kappa Q_0^* L Q_0 \\
        & = i Q_0^* D Q_1 - i Q_1^* D^* Q_0
\end{align*}
and
\begin{align*}
    \tau \frac{d}{dt} \|Q_1\|^2 & = \tau\left(Q_1^* \partial_t Q_1 + (\partial_t Q_1^*)Q_1\right) \\
        & = i \|Q_1\|^2 - i Q_1^* D Q_0 - i \|Q_1\|^2 + i Q_0^* D^* Q_1 \\
        & = - i Q_1^* D Q_0 + i Q_0^* D^* Q_1.
\end{align*}
Thus \eqref{discrete-mass-conservation} holds if $D^* = -D$.
\end{proof}
In Section \ref{sec:tests} we focus on problems with periodic boundary conditions
and take $D$ to be the Fourier spectral differentiation matrix.

\subsection{ImEx Runge-Kutta Time Discretization} \label{sec:imex}

An Implicit-Explicit (ImEx) RK method applied to the ODE system
\begin{equation}
    Q_t = f(Q) + g(Q)
\end{equation}
takes the form
\begin{subequations}
\label{Eq:ImExRK_Stage}
\begin{align}
\label{Eq:ImExRK_Stage_a}
    Q^{(i)} &= Q^n + \Delta t \left( \sum_{j=1}^{i-1} \tilde{a}_{ij} f(Q^{(j)}) + \sum_{j=1}^{i} a_{ij} g(Q^{(j)}) \right), \quad i = 1, 2, \dots, s,  \\
\label{Eq:ImExRK_Stage_b}
    Q^{n+1} &= Q^n + \Delta t \left( \sum_{j=1}^{s} \tilde{b}_j f(Q^{(j)}) + \sum_{j=1}^{s} b_j g(Q^{(j)}) \right),
\end{align}
\end{subequations}
where $Q^{n}$ denotes a numerical approximation of $Q$ at time $t_n$. This method is conveniently represented using the Butcher tableau:
\[
\begin{array}{c|c c}
    \tilde{c} & \tilde{A} \\
    \hline
    & \tilde{b}^T
\end{array}
\quad
\begin{array}{c|c c}
    c & A \\
    \hline
    & b^T
\end{array},
\]
where the strictly lower-triangular matrix $\tilde{A} = (\tilde{a}_{ij}) \in \mathbb{R}^{s \times s}$
corresponds to the explicit part, the lower-triangular matrix $A = (a_{ij}) \in \mathbb{R}^{s \times s}$
represents the implicit part, and the vectors $\tilde{c}, \tilde{b}, c,$ and
$b$ belong to $\mathbb{R}^s$. 

In order to apply such a method to the NLSH system \eqref{nlsH}, we must choose an
additive decomposition (or \emph{splitting} of the terms into the functions $f$
and $g$.  Ideally, the splitting should satisfy the following properties:
\begin{itemize}
    \item The implicit part $g(Q)$ allows for an efficient solution of the stage equations \eqref{Eq:ImExRK_Stage};
    \item the stable step size is independent of $\tau$
    \item the resulting discretization tends to a consistent discretization of the NLS equation as $\tau \to 0$ (i.e.,
        the \emph{asymptotic preserving property}, discussed below.
\end{itemize}
A promising splitting is obtained by handling the cubic terms explicitly and all other
(linear) terms implicitly:
\begin{equation}
\label{eqn: splitting}
\underbrace{\begin{pmatrix} \qzero \\ \qone \end{pmatrix}_t}_{= Q_t} = 
 \underbrace{\begin{pmatrix}  i \kappa \left|\qzero\right|^2 \qzero \\ 0 \end{pmatrix}}_{= f(Q)} + \underbrace{\begin{pmatrix} i\dqone \\  -i \tau^{-1}(\dqzero - \qone) \end{pmatrix}}_{= g(Q)}.
\end{equation}

We consider two types of ImEx-RK schemes: Type I (also known as Type A) and Type II (also known as Type CK).
The ImEx-RK method is classified as Type I if the matrix $A \in \mathbb{R}^{s \times s}$ is invertible. Type II ImEx-RK schemes are characterized by the following Butcher tableau:
\begin{align}
\label{GSA tableau}
\begin{array}{c|c c}
0 & 0 \\
\bm{\hat{\tilde{c}}} & \bm{\hat{\tilde{a}}} & \hat{\tilde{A}} \\
\hline
& \hat{\tilde{b}}_1 & {\pmb{\hat{\tilde{b}}}^T}
\end{array} \quad \begin{array}{c|c c}
0 & 0 \\
\bm{\hat{c}} & \bm{\hat{a}} & \hat{A} \\
\hline
& \hat{b}_1 & \bm{\hat{b}}^T
\end{array}.
\end{align}
Here, \( \bm{\hat{\tilde{c}}}, \bm{\hat{\tilde{a}}}, \bm{\hat{\tilde{b}}}, \bm{\hat{c}}, \bm{\hat{a}}, \) and \( \bm{\hat{b}} \) are vectors in \( \mathbb{R}^{s-1} \), \( \hat{\tilde{b}}_1 \) and \( \hat{b}_1 \) are constants in \( \mathbb{R} \), and \( \hat{\tilde{A}} \) and \(\hat{ A} \) are matrices in \( \mathbb{R}^{(s-1) \times (s-1)} \), where \( \hat{\tilde{A}} \) assumed to be invertible, indicating a diagonally implicit Runge-Kutta (DIRK) structure \cite{boscarino2024}.
An ImEx-RK method is said to be \textit{globally stiffly accurate} (GSA) if
\[
\tilde{a}_{si} = \tilde{b}_i \quad \text{and} \quad a_{si} = b_i, \quad i = 1, 2, \ldots, s.
\]

Furthermore, the initial data for NLSH \eqref{nlsH} is considered to be \textit{well-prepared} if
\( \qzero(x,0) = u \) and \( \qone(x,0) = D u \). 
In this section, where the solutions are treated as continuous in space, \( D \) represents the spatial derivative \( \partial_x \). 
In a fully discrete setting, \( D \) should instead be interpreted as the discrete differentiation operator used in the numerical scheme.

\subsection{Asymptotic-preserving time discretization} \label{sec:AP}
In the study of hyperbolic relaxation systems like the NLSH system \eqref{nlsH}, asymptotic-preserving (AP) schemes \cite{boscarino2024implicit,boscarino2024, shijin2001AP} provide a robust framework for handling stiff relaxation terms. These schemes ensure numerical stability and consistency across different stiffness regimes, allowing the numerical solution to naturally transition to the limiting problem as the relaxation parameter $\tau \to 0$. The NLSH system consists of convective and algebraic terms that depend on $\tau$, leading to extreme stiffness in the small $\tau$ limit. To mitigate stability constraints, stiff terms are integrated implicitly while non-stiff convection terms are treated explicitly. This implicit-explicit strategy enables stable and efficient numerical integration across multiple time scales. The key objective is to guarantee that as $\tau \to 0$, the numerical discretization naturally transitions to a stable and consistent discretization of the limiting problem, which, in this case, reduces to the NLS equation. We denote the continuous NLSH system as $P^\tau$ and its limiting problem as $P^0$. In developing AP schemes, we first focus on time discretization, deferring spatial discretization to later stages.

To formalize the AP property, let $P_h^\tau$ represent the numerical discretization of $P^\tau$, where $h$ denotes the discretization parameters, i.e. $h = (\Delta t, \Delta x)$. The AP property guarantees that for fixed $h$, the scheme $P_h^\tau$ provides, in the limit $\tau \to 0$, a consistent discretization of the limit problem $P^0$, denoted as $P_h^0$. This relationship is illustrated in Figure \ref{fig:AP}. The AP property can be defined more formally as follows:
\begin{definition}
A numerical discretization $P_h^\tau$ is said to be AP if its limiting discretization $P_h^0 = \lim_{\tau \to 0} P_h^\tau$ is a consistent and stable discretization of the continuous limit model $P^0$.
\end{definition}
\begin{figure}
    \centering
    \includegraphics[width= 6 cm]{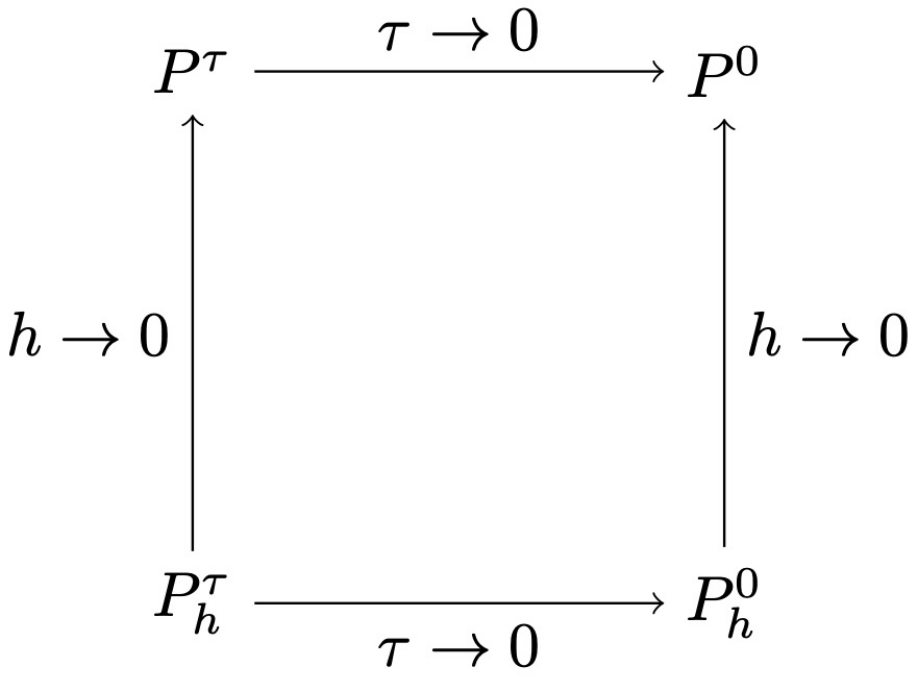}
    \caption{A schematic illustration of the AP property \cite{biswas2024kdvh}.}
    \label{fig:AP}
\end{figure}
An even strong property is that of \emph{asymptotic accuracy}:
\begin{definition}
A numerical discretization $P_h^\tau$ is said to be \emph{asymptotically accurate} (AA) if the limiting discretization $P_h^0 = \lim_{\tau \to 0} P_h^\tau$ maintains the temporal order of accuracy of the original scheme.
\end{definition}
\color{black}
Next, we define the AP property for the two components in our NLSH system, denoted by $Q^{n} = [q_0^n, q_1^n]^{T}$, which represents the numerical 
approximation of $Q$ at time $t_n$.
\begin{definition} An ImEx-RK method \eqref{Eq:ImExRK_Stage} applied to the
splitting \eqref{eqn: splitting} of the NLSH system is AP with respect to $\qzero$ if $\qzero^{n+1} \to u^{n+1}$ for $\tau
\to 0$, where $u^{n+1}$ is the numerical solution of
the NLS equation with splitting
\begin{equation}
\label{NLS_splitting} 
   u_t = \underbrace{
       i \kappa \left|u\right|^2 u }_{= f(u)} + \underbrace{
        i u_{xx} }_{= g(u)}
\end{equation}
and the same ImEx-RK method is AP for the auxiliary component $\qone$ if $\qone^{n+1} \to u^{n+1}_x$ for $\tau \to 0$.
\end{definition}

We now present the following results regarding the AP and AA properties of type I methods:

\begin{theorem} \label{thm:typeI}
An ImEx-RK method of type I applied to the splitting \eqref{eqn: splitting} of
the hyperbolic approximation of the NLS equation is always AP with respect to $\qzero$. For such a method, in the stiff limit $\tau \to 0$, we have
\begin{align}
    \qzero^{n+1} - u(t_{n+1}) &= \mathcal{O}(\Delta t^p), 
\end{align}
where $p$ is the order of the ImEx-RK method. Furthermore, if the method is assumed to be globally stiffly accurate, it is also AP for the auxiliary component $\qone$. In the stiff limit $\tau \to 0$, we have
\begin{equation}
    \qone^{n+1} - u_x(t_{n+1}) = \mathcal{O}(\Delta t^p).
\end{equation}
\end{theorem}

To demonstrate the AP property for a general ImEx-RK method of type II, an additional assumption concerning the well-preparedness of the initial data is necessary. 
According to the formulation in \cite{boscarino2024}, the initial data is said to be \textit{well-prepared} if it satisfies:
\begin{align}
\label{thm: well-prepared}
\qzero^0 = u^0 + \mathcal{O}(\tau) \quad \text{and} \quad \qone^0 = u_{x}^0 + \mathcal{O}(\tau).
\end{align}
For a general ImEx-RK method of type II, we prove the following result:
\begin{theorem}
\label{thm:typeII}
A globally stiffly accurate ImEx-RK method of type II, applied to the splitting \eqref{eqn: splitting} of the hyperbolic approximation of the NLS equation, together with the well-prepared initial data \eqref{thm: well-prepared}, is AP for both $\qzero$ and $\qone$. Furthermore, in the stiff limit $\tau \to 0$, the following error estimates apply to all components:
\[
\qzero^{n+1} - u(t_{n+1}) = \mathcal{O}(\Delta t^p), \quad 
\qone^{n+1} - u_x(t_{n+1}) = \mathcal{O}(\Delta t^p),
\]
\textit{where $p$ is the order of the ImEx-RK method.}
\end{theorem}
\begin{proof}
{\bf(Theorem \ref{thm:typeI})}
Let us denote the vectors of stage-solution components for the variables \( \qzero \) and \( \qone \) as follows: 
\( \bm{\qzero} = [\qzero^{(1)}, \qzero^{(2)}, \ldots, \qzero^{(s)}]^T \) and \( \bm{\qone} = [\qone^{(1)}, \qone^{(2)}, \ldots, \qone^{(s)}]^T \), and let $\bm{e} = [1,1,\dots,1]^T$ be the vector of ones in $\mathbb{R}^s$. For the NLSH system with the splitting \eqref{eqn: splitting}, the equation \eqref{Eq:ImExRK_Stage_a} in component form becomes
\begin{subequations}
\begin{align}
\label{eqn: type1 q0 component}
\bm{\qzero} &= \qzero^n \bm{e} + \Delta t \left( \tilde{A}\left(i\kappa\bm{|\qzero|}^2\bm{\qzero} 
\right) + A(i\bm{\dqone}) \right), \\
\label{eqn: type1 q1 component}
\bm{\qone} &= \qone^n \bm{e} + \frac{\Delta t}{\tau} \left( -i A(\bm{\dqzero}-\bm{\qone})\right).
\end{align}
\end{subequations}
Similarly, the final update of the solution can be written in components as
\begin{subequations}
\begin{align}
\label{eqn: type1 q0 stage}
\qzero^{n+1} &= \qzero^n + \Delta t \left( \bm{\tilde{b}^T}\left(i\kappa|\bm{\qzero}|^2\bm{\qzero}\right) 
 + \bm{b^T}(i\bm{\dqone})\right), \\
\label{eqn: type1 q1 stage}
\qone^{n+1} &= \qone^n + \frac{\Delta t}{\tau} \left(-i \bm{b^T}(\bm{\dqzero}-\bm{\qone})\right).
\end{align}
\end{subequations}
We assume there exist Hilbert expansions for \( \qzero^n \) and \( \qone^n \):
\begin{align}
\label{eq: hilbert}
\qzero^n &= (\qzero^n)_0 + \tau (\qzero^n)_1 + \tau^2 (\qzero^n)_2 + \cdots , \quad \qone^n = (\qone^n)_0 + \tau (\qone^n)_1 + \tau^2 (\qone^n)_2 + \cdots ,
\end{align}
and for the stage vectors \( \bm{\qzero}\) and \( \bm{\qone} \):
\begin{align}
\label{eq: hilbert stage vectors}
\bm{\qzero} &= (\bm{\qzero})_0 + \tau (\bm{\qzero})_1 + \tau^2 (\bm{\qzero})_2 + \cdots, \quad \bm{\qone} = (\bm{\qone})_0 + \tau (\bm{\qone})_1 + \tau^2 (\bm{\qone})_2 + \cdots.
\end{align}
We substitute these expansions into the stage equations and equate the leading-order terms in powers of $\tau$. This process applied to the stage equations \eqref{eqn: type1 q1 component} results in:
\begin{equation}
    \mathcal{O}(\tau^{-1}): \quad
        A\left(\bm{\partial_x}\qzzv-\qozv\right) = 0 .
\end{equation}
Since $A$ is invertible, it follows that
\begin{equation}
\label{low order results}
   \qozv = \bm{\partial_x}\qzzv.
\end{equation}
The leading-order term in the expansion of \eqref{eqn: type1 q0 component} gives:
\begin{equation}
    \mathcal{O}(\tau^{0}): \quad
  \qzzv = (\qzero^n)_0 \bm{e} + \Delta t \left( i\kappa \tilde{A}|(\bm{\qzero})_0|^2\qzzv  + i A\bm{\partial_x}\qozv \right).
\end{equation}
Using \eqref{low order results} in the above equation we get 
\begin{equation}
\label{leading_order_terms}
    \qzzv = (\qzero^n)_0 \bm{e} + \Delta t \left( i\kappa \tilde{A}|(\bm{\qzero})_0|^2\qzzv  + i A\bm{\partial_{xx}}\qzzv \right).
\end{equation}
Using equation \eqref{low order results} again in the leading-order term of the solution update for the first equation yields:
\begin{equation}
\label{leading_order_stage_terms}
  (\qzero^{n+1})_0 = (\qzero^n)_0 + \Delta t \left( i\kappa\bm{\tilde{b}^T}|(\bm{\qzero})_0|^2\qzzv 
 + i \bm{b^T}\bm{\partial_{xx}}\qzzv \right).
\end{equation}
In the limit as $\tau \to 0$, the numerical scheme becomes:
\begin{align}
    \qzzv &= (\qzero^n)_0 \bm{e} + \Delta t \left( i\kappa \tilde{A}|(\bm{\qzero})_0|^2\qzzv  + i A \bm{\partial_{xx}} \qzzv\right),\\
  (\qzero^{n+1})_0 &= (\qzero^n)_0 + \Delta t \left( i\kappa\bm{\tilde{b}^T}|(\bm{\qzero})_0|^2\qzzv 
 + i \bm{b^T}\bm{\partial_{xx}}\qzzv \right).
\end{align}
This corresponds exactly to the numerical scheme derived from the time-stepping method \eqref{Eq:ImExRK_Stage} of type I when applied to the NLS equation \eqref{NLS_splitting}, where the term $f(u)$ is handled explicitly, while $g(u)$ is treated implicitly. Consequently, we also obtain: \begin{align}
    (\qzero^{n+1})_0 - u(t_{n+1}) &= \mathcal{O}(\Delta t^p),
\end{align} where $p$ is the order of the ImEx-RK method.
Furthermore, if the method is GSA, we can establish the AP property for the auxiliary variable $\qone$. To demonstrate this, we consider equation \eqref{eqn: type1 q1 component}. By utilizing the invertibility of $A$, we get:
\begin{align}
A^{-1}\left(\bm{\qone} - \qone^n \bm{e}\right)  &=  \frac{\Delta t}{\tau} \left( -i(\bm{\dqzero}-\bm{\qone})\right).
\end{align}
Substituting this expression into the update rule for the auxiliary component \eqref{eqn: type1 q1 stage}, we obtain:
\begin{align}
    \qone^{n+1} &= \qone^n + \bm{b}^T A^{-1} (\bm{\qone} - \qone^n \bm{e}) = \bm{b}^T A^{-1} \bm{\qone} + (1 - \bm{b}^T A^{-1} \bm{e}) \qone^n.
\end{align}
Since the method is assumed to be GSA, the stiff accuracy of \( A \) ensures that \( \bm{b}^T A^{-1} = [0, 0, \dots, 1] \) in \( \mathbb{R}^s \), which implies \( 1 - \bm{b}^T A^{-1} \bm{e} = 0 \). Using this property, we obtain:
\begin{align}
    \qone^{n+1} &= \bm{b}^T A^{-1} \bm{\qone} = \qone^{(s)}.
\end{align}
In the limit as \( \tau \to 0 \), we obtain $ (\qone^{n+1})_0 = (\qone^{(s)})_0$. Using equation \eqref{low order results}, 
we can express $(\qone^{n+1})_0 = \partial_{x}(\qzero^{(s)})_0$. By the GSA property of the scheme, 
it follows from equations \eqref{leading_order_terms} and \eqref{leading_order_stage_terms} that \( (\qzero^{(s)})_0 =(\qzero^{n+1})_0 \), leading to 
\( (\qone^{n+1})_0 = \partial _x( \qzero^{n+1})_0 \). This confirms the AP property of the type I ImEx-RK method for the auxiliary components. In the stiff limit \( \tau \to 0 \), the following error estimate holds for \( \qone \):
\begin{equation}
    \qone^{n+1} - u_x(t_{n+1}) = \mathcal{O}(\Delta t^p).
\end{equation}
\end{proof}

Next, we analyze the asymptotic-preserving (AP) property of the ARS(1,1,1) method when applied to the splitting in \eqref{eqn: splitting}. The ARS(1,1,1) method updates the solution from time step \( t^n \) to \( t^{n+1} \) as follows:
\begin{equation}
Q^{n+1} = Q^n + \Delta t \left(f \left(Q^n\right) + g\left(Q^{n+1}\right)\right) ,
\end{equation}
where the update components are given by:
\begin{subequations}
\label{eqn: discretized-sys}
\begin{align}
\label{eqn: q0}
\qzero^{n+1} &= \qzero^n + i \Delta t \left( \kappa|\qzero^n|^2\qzero^n + \dqone^{n+1} \right),  \\
\label{eqn: q1}
\qone^{n+1} &= \qone^n -i \frac{\Delta t}{\tau} \left(\dqzero^{n+1} - \qone^{n+1} \right).
\end{align}
\end{subequations}
To analyze the AP property, we insert Hilbert expansions \eqref{eq: hilbert} into the update equations \eqref{eqn: discretized-sys}. Analyzing the leading-order terms with respect to \( \tau \), we find from \eqref{eqn: q1} that
\begin{equation}
\label{eqn: leading-order-terms}
\mathcal{O} \left( \tau^{-1} \right) : \quad (\qone^{n+1})_0 = \partial_x\left( \qzero^{n+1} \right)_0 . 
\end{equation}
The order \( \mathcal{O}(\tau^0) \) terms in \eqref{eqn: q0} yield

\begin{align}
\mathcal{O} \left( \tau^0 \right) : \quad (\qzero^{n+1})_0 = (\qzero^n)_0 + i \Delta t \left( \kappa|(\qzero^n)_0|^2(\qzero^n)_0 + \partial_x(\qone^{n+1})_0 \right).
\end{align}
Replacing \( (\qone^{n+1})_0 \) with its value from \eqref{eqn: leading-order-terms}, we obtain

\begin{align}
 \frac{(\qzero^{n+1})_0 - (\qzero^n)_0 }{\Delta t} &=  i  \left( \kappa|(\qzero^n)_0|^2(\qzero^n)_0 + \partial_{xx}(\qzero^{n+1})_0 \right) .
\end{align}

In the limit \( \tau \to 0 \), these equations yield a discretization of the NLS equation using the ARS(1,1,1) method, thereby demonstrating the AP property of the ARS(1,1,1) method for the splitting \eqref{eqn: splitting}.

\begin{proof}
{\bf (Theorem \ref{thm:typeII})}
We now prove the AP property for the CK-type globally stiffly accurate (GSA) ARK methods applied to the splitting \eqref{eqn: splitting} of the NLS equation.
Starting from the solution \( Q^n \) at time \( t_n \), a CK-type SA ARK method applied to the system 
\[
Q_t = f(Q) + g(Q)
\]
computes the approximate solution at time \( t_{n+1} \) as follows:
\begin{subequations}
\begin{align}
Q^{(1)} &= Q^n, \\
Q^{(i)} &= Q^n + \Delta t \left( \hat{\tilde{a}}_{i1} f\left( Q^{n} \right) + \sum_{j=2}^{i-1} \hat{\tilde{a}}_{ij} f\left( Q^{(j)} \right) + \hat{a}_{i1} g\left( Q^{(n)} \right) + \sum_{j=2}^{i} \hat{a}_{ij} g\left( Q^{(j)} \right) \right), \quad i = 2,3,\ldots,s, \\
Q^{n+1} &= Q^n + \Delta t \left( \hat{\tilde{b}}_1 f\left( Q^{n} \right) + \sum_{j=2}^s \hat{\tilde{b}}_j f\left( Q^{(j)} \right) + \hat{b}_1 g\left( Q^{n} \right) + \sum_{j=2}^s \hat{b}_j g\left( Q^{(j)} \right) \right).
\end{align}
\end{subequations}
By introducing the notation \( \bm{Q} = [Q^{(2)}, Q^{(3)}, \ldots, Q^{(s)}]^T \), the identity matrix \( I \in \mathbb{R}^{2 \times 2} \), and the vector of ones \( \bm{e} \in \mathbb{R}^{s-1} \), we can express this compactly as follows:
\begin{subequations}
\begin{align}
\label{eqn: compact form}
\bm{Q} &= \bm{e} \otimes Q^n + \Delta t \left( \bm{\hat{\tilde{a}}} \otimes f(Q^n) + (\hat{\tilde{A}} \otimes I) f(\bm{Q}) + \bm{\hat{a}} \otimes g(Q^n) + (\hat{A} \otimes I) g(\bm{Q}) \right), \\
\label{eqn: compact update}
Q^{n+1} &= Q^n + \Delta t \left( \hat{\tilde{b}}_1 f(Q^n) + (\bm{\hat{\tilde{b}}}^T \otimes I) f(\bm{Q}) + \hat{b}_1 g(Q^n) + (\bm{\hat{b}}^T \otimes I) g(\bm{Q}) \right).
\end{align}
\end{subequations}
Let us denote the vectors of stage-solution components for the variables \( \qzero \) and \( \qone \) as follows: 
\\
\( \bm{\qzero} = [\qzero^{(2)}, \qzero^{(3)}, \ldots, \qzero^{(s)}]^T \) and \( \bm{\qone} = [\qone^{(2)}, \qone^{(3)}, \ldots, \qone^{(s)}]^T \). For the NLSH system with the splitting defined in \eqref{eqn: splitting}, the equation \eqref{eqn: compact form} in component form is expressed as:
\begin{subequations}
\begin{align}
\label{eqn: GSA q0 component}
\bm{\qzero} &= \qzero^n \bm{e} + i \Delta t \left( \kappa|\qzero^n|^2\qzero^n \bm{\hat{\tilde{a}}} + \kappa \hat{\tilde{A}}(|\bm{\qzero}|^2\bm{\qzero}) + \dqone^n \bm{\hat{a}} + \hat{A}(\bm{\dqone})\right), \\
\label{eqn: GSA q1 component}
\bm{\qone} &= \qone^n \bm{e} +i \frac{\Delta t}{\tau} \left(-(\dqzero^n - \qone^n) \bm{\hat{a}} - \hat{A}(\bm{\dqzero} - \bm{\qone})\right).
\end{align}
\end{subequations}
Similarly, the update of the solution from \eqref{eqn: compact update} can be expressed in component form as follows:
\begin{subequations}
\begin{align}
\qzero^{n+1} &= \qzero^n + i \Delta t \left( \kappa|\qzero^n|^2\qzero^n\hat{\tilde{b}}_1 + \kappa {\pmb{\hat{\tilde{b}}}^T}(|\bm{\qzero}|^2\bm{\qzero}) + \dqone^n \hat{b}_1 + \bm{\hat{b}}^T(\bm{\dqone}) \right), \\
\qone^{n+1} &= \qone^n + i \frac{\Delta t}{\tau} \left(-(\dqzero^n - \qone^n) \hat{b}_1  - \bm{\hat{b}^T}(\bm{\dqzero} - \bm{\qone})\right).
\end{align}
\end{subequations}
Next, we substitute the Hilbert expansions \eqref{eq: hilbert}, along with the expansions for the stage vectors \eqref{eq: hilbert stage vectors}, into the stage equations and equate the leading-order terms. Applying this to the stage equation \eqref{eqn: GSA q1 component} yields:
\begin{align}
\mathcal{O} \left( \tau^{-1} \right) : \quad (\partial_x(\qzero^n)_0 - (\qone^n)_0) \bm{\hat{a}} + \hat{A}(\bm{\partial_x}(\bm{\qzero})_0 -( \bm{\qone})_0) = \bm{0}.
\end{align}
By the well-preparedness of the initial data at time $t_n$, we have   \( (\qone^n)_0 = \partial_x(\qzero^n)_0 \), and given that \( \hat{A} \) is invertible, we can derive the following relationship:
\begin{align}
\label{q0 p0}
\bm{(\qone)}_0 = \bm{\partial_x}(\bm{\qzero})_0.
\end{align}
The leading-order terms in the expansion of \eqref{eqn: GSA q0 component} yield:

\begin{align}
\mathcal{O} \left( \tau^0 \right) : \quad (\bm{\qzero})_0 = (\qzero^n)_0 \bm{e} + i \Delta t \left( \kappa|(\qzero^n)_0|^2(\qzero^n)_0 \bm{\hat{\tilde{a}}} + \kappa \hat{\tilde{A}}|(\bm{\qzero})_0|^2(\bm{\qzero})_0 + \partial_x(\qone^n)_0 \bm{\hat{a}} + \hat{A}\bm{\partial_{x}}\qozv\right).
\end{align}

Substituting \( \bm{(\qone)}_0 = \bm{\partial_x}(\bm{\qzero})_0 \) into the equation above, and the well-preparedness of the initial data at time $t_n$, we obtain:

\begin{align}
\label{components update}
(\bm{\qzero})_0 = (\qzero^n)_0 \bm{e} + i \Delta t \left( \kappa|(\qzero^n)_0|^2(\qzero^n)_0 \bm{\hat{\tilde{a}}} + \kappa \hat{\tilde{A}}|(\bm{\qzero})_0|^2(\bm{\qzero})_0 + \partial_{xx}(\qzero^n)_0 \bm{\hat{a}} + \hat{A}\bm{\partial_{xx}}(\bm{\qzero})_0\right).
\end{align}

The solution updates for the first equation in the leading-order term can then be expressed as:

\begin{align}
\label{solutions update}
(\qzero^{n+1})_0 &= (\qzero^n)_0 + i \Delta t \left( \kappa|(\qzero^n)_0|^2(\qzero^n)_0\hat{\tilde{b}}_1 + \kappa {\pmb{\hat{\tilde{b}}}^T}(|(\bm{\qzero})_0|^2(\bm{\qzero})_0) + \partial_x(\qone^n)_0 \hat{b}_1 + \bm{\hat{b}}^T\bm{\partial_{xx}}(\bm{\qzero})_0\right).
\end{align}

As \(\tau \to 0\), we have \(\qzero^n = (\qzero^n)_0\) and \(\bm{\qzero} = \qzzv\). Therefore, the numerical scheme simplifies to:
\begin{subequations}
\begin{align}
\bm{\qzero} &= \qzero^n \bm{e} + i \Delta t \left( \kappa|\qzero^n|^2\qzero^n \bm{\hat{\tilde{a}}} + \kappa \hat{\tilde{A}}(|\bm{\qzero}|^2\bm{\qzero}) + \ddqzero^n \bm{\hat{a}} + \hat{A}\bm{\ddqzero}\right), \\
\qzero^{n+1} &= \qzero^n + i \Delta t \left( \kappa|\qzero^n|^2\qzero^n\hat{\tilde{b}}_1 + \kappa {\pmb{\hat{\tilde{b}}}^T}(|\bm{\qzero}|^2\bm{\qzero}) + \ddqzero^n \hat{b}_1 + \bm{\hat{b}}^T \bm{\ddqzero} \right).
\end{align}
\end{subequations}
This formulation aligns with the numerical scheme derived from the CK-type GSA ARK methods, as outlined in \eqref{GSA tableau}, when applied to the NLS equation. To ensure the accuracy of the scheme for the auxiliary variable \(\qone\), we use equation \eqref{eqn: GSA q1 component}. Utilizing the invertibility of \(\hat{A}\), we derive:
\begin{align}
i \frac{\Delta t}{\tau} (\bm{\dqzero} - \bm{\qone}) &= - \hat{A}^{-1} (\bm{\qone} - \qone^n \bm{e}) -i \frac{\Delta t}{\tau} (\dqzero^n - \qone^n) \hat{A}^{-1} \bm{\hat{a}}.
\end{align}
Substituting these into the update equation yields:
\begin{subequations}
\begin{align}
\qone^{n+1} &= \qone^n -i \frac{\Delta t}{\tau} (\dqzero^n - \qone^n) \hat{b}_1  + \bm{\hat{b}^T}\left(\hat{A}^{-1} (\bm{\qone} - \qone^n \bm{e}) +i \frac{\Delta t}{\tau} (\dqzero^n - \qone^n) \hat{A}^{-1} \bm{\hat{a}}\right) \\ 
&= \bm{\hat{b}^T}\hat{A}^{-1} \bm{\qone} + \left(1 - \bm{\hat{b}^T}\hat{A}^{-1} \bm{e}\right)\qone^n +i \frac{\Delta t}{\tau} \left(-\hat{b}_1 + \bm{\hat{b}^T}\hat{A}^{-1} \bm{\hat{a}}\right)(\dqzero^n - \qone^n).
\end{align}
\end{subequations}
Assuming the method is GSA, the stiff accuracy condition \(\bm{\hat{b}}^T \hat{A}^{-1} = [0, \dots, 0, 1]\) in \(\mathbb{R}^{s-1}\) implies that \(1 - \bm{\hat{b}}^T \hat{A}^{-1} \bm{e} = 0\) and \(\bm{\hat{b}}^T \hat{A}^{-1} \bm{\hat{a}} = \hat{b}_1 \). This leads us to:
\begin{align}
\qone^{n+1} &= \bm{\hat{b}}^T \hat{A}^{-1} \bm{\qone} = \qone^{(s)}.
\end{align}
As \(\tau \to 0\), we find that \((\qone^{n+1})_0 = (\qone^{(s)})_0\). By employing the expressions derived in \eqref{q0 p0}, we can express  \((\qone^{n+1})_0 =\partial_x(\qzero^{(s)})_0 \). Due to the GSA property of the scheme, it follows from \eqref{components update} and \eqref{solutions update}  that \((\qzero^{(s)})_0 = (\qzero^{n+1})_0\). Consequently, this leads us to \((\qone^{n+1})_0 = \partial_x(\qzero^{n+1})_0\). Thus, we establish the asymptotic-preserving property of the method.
\end{proof}

Note that the initial data is assumed to be well-prepared at time $ t_n $. This assumption is justified by noting that if the initial condition satisfies the well-preparedness condition specified in \eqref{thm: well-prepared}, then \eqref{thm:typeII} (for $n = 0 $) guarantees that the solution remains well-prepared at step $ n = 1 $. By induction, it follows that the auxiliary component remains on the local equilibrium manifold for all subsequent time steps. The GSA property is essential to ensure that the auxiliary variable lies on the local equilibrium manifold, which in turn preserves the accuracy of $\qzero$ and the auxiliary variable $\qone$.

\subsection{Fully-discrete mass conservation via temporal relaxation} \label{sec:time-relaxation}
In this section we describe how fully-discrete mass conservation can be achieved
using the technique of relaxation in time~\cite{ketcheson2019relaxation, ranocha2020general, ranocha2020relaxation}.
Note that the use of the term \emph{relaxation} here is entirely separate from the hyperbolic
relaxation of the NLS equation.

In this work, we focus on preserving a single invariant at the fully discrete
level using ImEx RK methods. Applications of the relaxation technique with ImEx methods for conserving multiple invariants can be found in 
\cite{biswas2023multiple,biswas2024accurate}. For purposes of the time relaxation technique,
all details of the time discretization are unimportant; we only need to consider
that we have a map from $Q^n$ to $Q^{n+1}$ and that this map is locally consistent
to at least second order in $\Delta t$.

We define the new approximation, dependent on the parameter $\gamma_n$:
\begin{equation}
    Q(t_n + \gamma_n \Delta t) \approx Q^{n+1}_{\gamma_n} = Q^n + \gamma_n \Delta t (Q^{n+1} - Q^n),
\end{equation}  
and we choose the value of $\gamma_n$ so as to enforce the conservation condition
\begin{equation}
    \|Q^{n+1}_{\gamma_n}\|^2 = \|Q^n\|^2.
\end{equation}
It can be shown that a suitable value of $\gamma_n$ always exists; in fact in
the present setting it is given by:
\begin{equation}
    \gamma_n = \frac{\sum_{\omega=0}^{1}\sum_{i,j=1}^{s}\left\{\left(\tilde{b}_{i} f^i +  b_{i} g^i\right)_\omega,\left(\tilde{a}_{ij} f^j +  a_{ij} g^j\right)_\omega\right\}}
    {\|\sum_{i=1}^{s}(\tilde{b}_{i} f^i +  b_{i} g^i)_0\|^2+\tau\|\sum_{i=1}^{s}(\tilde{b}_{i} f^i +  b_{i} g^i)_1\|^2}
\end{equation}
where $f^i=f(Q^{(i)})$, $g^i=g(Q^{(i)})$ , $\{A,B\}\equiv(\langle A,B\rangle+\langle A,B\rangle)$ and $\langle A,B\rangle\equiv\int (A^*B)dx$
Furthermore, this value is sufficiently close to unity 
(\( \gamma_n = 1 + \mathcal{O}(\Delta t^{p-1}) \)) so that, if one
interprets the solution $Q^{n+1}_{\gamma_n}$ as an approximation at time
$(t_n + \gamma_n \Delta t)$, then the order of accuracy of the original
time stepping scheme is preserved.

\section{Numerical tests} \label{sec:tests}
In this section, we present numerical results to verify the theoretical
findings discussed in the previous sections. The NLSH system is
semi-discretized in space using the Fourier spectral differentiation matrix
$D$, as described in Section~\ref{sec:space-disc}. The NLS 
equation is also semi-discretized using the Fourier differentiation matrix for
the second derivative term. To integrate both semi-discretized systems in time,
we employ ImEx Runge–Kutta schemes with different properties, specifically
those known to be stiffly 
accurate (SA) and/or First Same As Last (FSAL). A scheme is called FSAL if its
explicit part satisfies $\tilde{a}_{si} = \tilde{b}_i$, and it is called SA if
the DIRK method satisfies $a_{si} = b_i$.  The methods used are listed in Table
\ref{tab:methods}.

\begin{table}
\centering
\caption{ImEx Runge-Kutta methods and some of their properties.  See the text for
an explanation of the abbreviated terms.}
\begin{tabular}{|l|c|c|c|c|c|}
\hline
Method                                    & Order & Type & SA  & FSAL & GSA \\ \hline
SSP2-ImEx(3,3,2) \cite{boscarino2024}     & 2     & I    & Yes & No   & No \\ \hline
SSP3-ImEx(3,4,3) \cite{boscarino2024}     & 3     & I    & No  & No   & No \\ \hline
AGSA(3,4,2)      \cite{boscarino2024}     & 2     & I    & Yes & Yes  & No \\ \hline
ARS(4,4,3) \cite{ascher1997implicit}      & 3     & II   & Yes & Yes  & Yes \\ \hline
ARK3(2)4L[2]SA \cite{kennedy2003additive} & 3     & II   & Yes & No   & No \\ \hline
ARK4(3)6L[2]SA \cite{kennedy2003additive} & 4     & II   & Yes & No   & No \\ \hline
\end{tabular}
\label{tab:methods}
\end{table}

We start by studying the convergence of numerical solutions of the NLSH system
to those of the NLS equation: qualitatively by comparing solution profiles, and
quantitatively in terms of the AP property in
Subsection~\ref{subsec:numeric_AP}. Then, we numerically investigate the AA
property in 
Subsection~\ref{subsec:AAP}. In Subsection~\ref{subsec:relax}, we show that the
relaxation techniques introduced in~\cite{ketcheson2025approximation}, which
improve the conservation properties of NLS solutions, can also be used to
conserve the corresponding conserved quantities in the NLSH system. 
Finally, in Subsection~\ref{subsec:DSW}, we consider a smoothed Riemann problem
for the defocusing NLS equation ($\kappa<0$) comparing results for the NLSH system
with the NLS solution.

\subsection{Verification of Asymptotic Preserving Property}
\label{subsec:numeric_AP}
We first consider the initial value problem with $u(x,t=0) = \sech(x)$, with
periodic boundary conditions on domain $x\in[-16,16]$ and $\kappa = 8$.  The solution in this case
consists of two solitons in a bound state that interact in a time-periodic
manner.  Solutions of the NLS and NLSH equation for this case, as well as the
case with $\kappa=18$ (which produces a 3-soliton bound state)
are shown in Figure \ref{fig:solitons_23}. 
As expected, the NLSH solution tends to that of NLS as $\tau \to 0$.
We also see that, for a given value of $\tau$, larger differences
are visible in the 3-soliton case.  This is likely due to the presence
of larger gradients (and larger values of higher derivatives) in the 3-soliton solution.

\begin{figure}
    \centering
    \includegraphics[width=\linewidth]{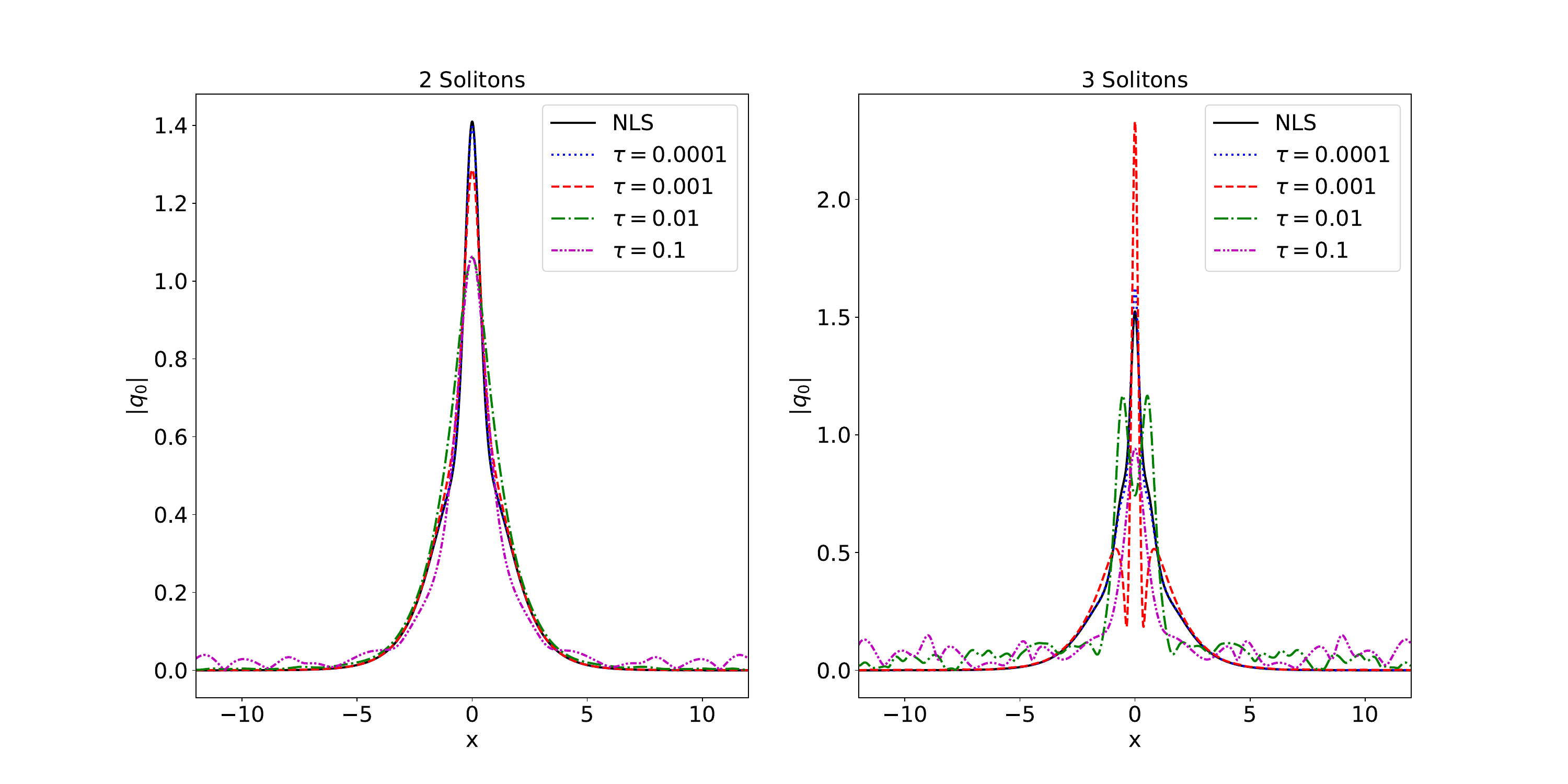}
    \caption{Bound state soliton solutions for 2 (left) and 3 (right) solitons.
    As $\tau\to 0$ the NLSH solution converges to NLS solution.}
    \label{fig:solitons_23}
\end{figure}

We perform numerical experiments to further investigate the numerical convergence of hyperbolized system to the original system. For these experiments, 
we consider the spatial domain of $x\in[-16,16]$ discretized with $2^{11}$ grid points. We evolve both the NLS system and NLSH system using the range 
of ImEx schemes described above, and look at the difference between the numerical solutions of NLS and NLSH (referred to 
as \emph{hyperbolization error}) at $T=5$. Specifically, we compute $\|  u-q_0\|_2 $ and $\| Du-q_1 \|_2 $, where $D$ is
the discrete differentiation operator used in our spatial discretization. Note that the errors and estimated order of
convergence (EOC) presented in tables \ref{tab:AP-SSP2}-\ref{tab:AP-ARK4} are calculated using the $\mathcal{L}_2$ norm of the
hyperbolization error.

\subsubsection{AP results for type I ImEx-RK methods}
In Table \ref{tab:AP-SSP2}, we present the errors for SSP2-ImEx(3,3,2). The implicit part in this scheme is SA but the explicit part is not FSAL, so the overall method is not GSA. Our theory from Section \ref{sec:AP} guarantees convergence for $q_0$, but not for $q_1$. 
Interestingly, we do observe some convergence for $q_1$, although the rate is not consistent and falls off at small values of $\tau$.
Next we test SSP3-ImEx(3,4,3). This scheme is neither SA nor FSAL and hence lacking the GSA property,
so again the theory only guarantees convergence of $q_0$.  In this case the convergence for $q_1$ falls off much earlier.

Table \ref{tab:AP-AGSA} shows results for the scheme AGSA(3,4,2). This method satisfies the GSA property, so both $q_0$ and
$q_1$ are gauranteed to converge, which is what we observe.

\begin{table}
\centering
\caption{Errors and estimated order of convergence for the 2-soliton problem with SSP2-IMEX(3,3,2).}
\begin{tabular}{|c|c|c|c|c|}
\hline
$\tau$ & $|u-q_0|_2$ & EOC $q_0$ & $|Du-q_1|_2$ & EOC $q_1$ \\ 
\hline
0.01& 4.411e-01 &   & 4.966e-01 &  \\
\hline
0.0001& 7.697e-03 & 0.879 & 1.301e-02 & 0.791 \\
\hline
1e-06& 7.744e-05 & 0.999 & 1.318e-04 & 0.997 \\
\hline
1e-08& 7.745e-07 & 1.0 & 1.421e-06 & 0.984 \\
\hline
1e-10& 7.745e-09 & 1.0 & 2.747e-07 & 0.357 \\
\hline
\end{tabular}
\label{tab:AP-SSP2}
\end{table}

\begin{table}
\centering
\caption{Errors and estimated order of convergence for the 2-soliton problem with SSP3-IMEX(3,4,3).}
\begin{tabular}{|c|c|c|c|c|}
\hline
$\tau$ & $|u-q_0|_2$ & EOC $q_0$ & $|Du-q_1|_2$ & EOC $q_1$ \\ 
\hline
0.01& 4.411e-01 &   & 4.966e-01 &  \\
\hline
0.0001& 7.697e-03 & 0.879 & 1.301e-02 & 0.791 \\
\hline
1e-06& 7.744e-05 & 0.999 & 1.347e-04 & 0.992 \\
\hline
1e-08& 7.744e-07 & 1.0 & 2.081e-04 & -0.095 \\
\hline
1e-10& 7.744e-09 & 1.0 & 2.092e-04 & -0.001 \\
\hline
\end{tabular}
\label{tab:AP-SSP3}
\end{table}


\begin{table}
\centering
\caption{Errors and estimated order of convergence for the 2-soliton problem with AGSA(3,4,2).}
\begin{tabular}{|c|c|c|c|c|}
\hline
$\tau$ & $|u-q_0|_2$ & EOC $q_0$ & $|Du-q_1|_2$ & EOC $q_1$ \\
\hline
0.01& 4.411e-01 &   & 4.966e-01 &  \\
\hline
0.0001& 7.651e-03 & 0.88 & 1.292e-02 & 0.792 \\
\hline
1e-06& 7.748e-05 & 0.997 & 1.318e-04 & 0.996 \\
\hline
1e-08& 7.750e-07 & 1.0 & 1.319e-06 & 1.0 \\
\hline
1e-10& 7.750e-09 & 1.0 & 1.316e-08 & 1.0 \\
\hline
\end{tabular}
\label{tab:AP-AGSA}
\end{table}

\subsubsection{ Type II ImEx-RK methods}
Now we present the results for type II methods.  Table \ref{tab:AP-ARS} shows
the convergence rates for ARS(4,4,3). This method is GSA and, as expected,
shows convergences for both components ($q_0$ and $q_1$). The next two tables
(\ref{tab:AP-ARK3} and \ref{tab:AP-ARK4}) present results for 
ARK3(2)4L[2]SA and ARK4(3)6L[2]SA, respectively.  Both of these methods are not GSA but are SA
in their implicit part. While we do not have any theoretical guarantees for these
methods, it appears that the higher order method (ARK4(3)6L[2]SA) gives convergence for
$q_0$, while the convergence of $q_1$ begins to stagnate at small values of $\tau$.
Results for the ARK3(2)4L[2]SA, shown in Table \ref{tab:AP-ARK3}, 
are nearly identical to those for SSP2-IMEX(3,3,2) in Table \ref{tab:AP-SSP2}.

It is notable that Table 1 and Table 5 show almost identical results.
Both of the methods in question are SA but not FSAL.  We hypothesize that,
even though the Type II method is one order higher than the Type I method, 
on a sufficiently fine numerical mesh
the hyperbolization errors in $q_0$ and $q_1$ are not significantly affected by
the choice betwween these two methods.

\begin{table}
\centering
\caption{Errors and estimated order of convergence for the 2-soliton problem with ARS(4,4,3).}
\begin{tabular}{|c|c|c|c|c|}
\hline
$\tau$ & $|u-q_0|_2$ & EOC $q_0$ & $|Du-q_1|_2$ & EOC $q_1$ \\ 
\hline
0.01& 4.411e-01 &   & 4.966e-01 &  \\
\hline
0.0001& 7.697e-03 & 0.879 & 1.301e-02 & 0.791 \\
\hline
1e-06& 7.744e-05 & 0.999 & 1.317e-04 & 0.997 \\
\hline
1e-08& 7.744e-07 & 1.0 & 1.318e-06 & 1.0 \\
\hline
1e-10& 7.744e-09 & 1.0 & 1.317e-08 & 1.0 \\
\hline
\end{tabular}
\label{tab:AP-ARS}
\end{table}
\begin{table}
\centering
\caption{Errors and estimated order of convergence for the 2-soliton problem with ARK3(2)4L[2]SA.}
\begin{tabular}{|c|c|c|c|c|}
\hline
$\tau$ & $|u-q_0|_2$ & EOC $q_0$ & $|Du-q_1|_2$ & EOC $q_1$ \\ 
\hline
0.01& 4.411e-01 &   & 4.966e-01 &  \\
\hline
0.0001& 7.697e-03 & 0.879 & 1.301e-02 & 0.791 \\
\hline
1e-06& 7.744e-05 & 0.999 & 1.318e-04 & 0.997 \\
\hline
1e-08& 7.744e-07 & 1.0 & 1.425e-06 & 0.983 \\
\hline
1e-10& 7.744e-09 & 1.0 & 2.838e-07 & 0.35 \\
\hline
\end{tabular}
\label{tab:AP-ARK3}
\end{table}

\begin{table}
\centering
\caption{Errors and estimated order of convergence for the 2-soliton problem  with ARK4(3)6L[2]SA.}
\begin{tabular}{|c|c|c|c|c|}
\hline
$\tau$ & $|u-q_0|_2$ & EOC $q_0$ & $|Du-q_1|_2$ & EOC $q_1$ \\
\hline
0.01& 4.411e-01 &   & 4.966e-01 &  \\
\hline
0.0001& 7.697e-03 & 0.879 & 1.300e-02 & 0.791 \\
\hline
1e-06& 7.744e-05 & 0.999 & 1.317e-04 & 0.997 \\
\hline
1e-08& 7.744e-07 & 1.0 & 1.318e-06 & 1.0 \\
\hline
1e-10& 7.744e-09 & 1.0 & 1.317e-08 & 1.0 \\
\hline
\end{tabular}
\label{tab:AP-ARK4}
\end{table}


\subsection{Asymptotic Accuracy}
\label{subsec:AAP}
In this section, we numerically study the asymptotic accuracy of different ImEx schemes. We use ground state solutions of the NLSH system ($q_g$ in \eqref{eq:translation_rotation_ansatz}) as 
test cases here, in particular we consider 
$\kappa=1$, $\mu=3$, and $\qt^+$ (see \eqref{eq:explicit_solitary_NLSH}) with $K_1=0$. 
We take the domain $[-5\pi,5\pi]$ discretized with $2^{11}$ points in space, and
study convergence for the following schemes: AGSA(3,4,2), SSP3-ImEx(3,4,3), ARK3(2)4L[2]SA, and ARS(4,4,3). 
We compare the numerical solution of the NLSH system with the exact solution,
denoted by $q_j^{\text{ex}}$.
In Figure \ref{fig:exact_AA}, we plot the norm of the differences at $t=5$, for $\tau=10^{-2}$ and 
$\tau=10^{-5}$. 
For smaller values of $\tau$, we obtained results almost indistinguishable
from those shown with $\tau=10^{-5}$.

\begin{figure}
    \centering
    \includegraphics[width=1.\linewidth]{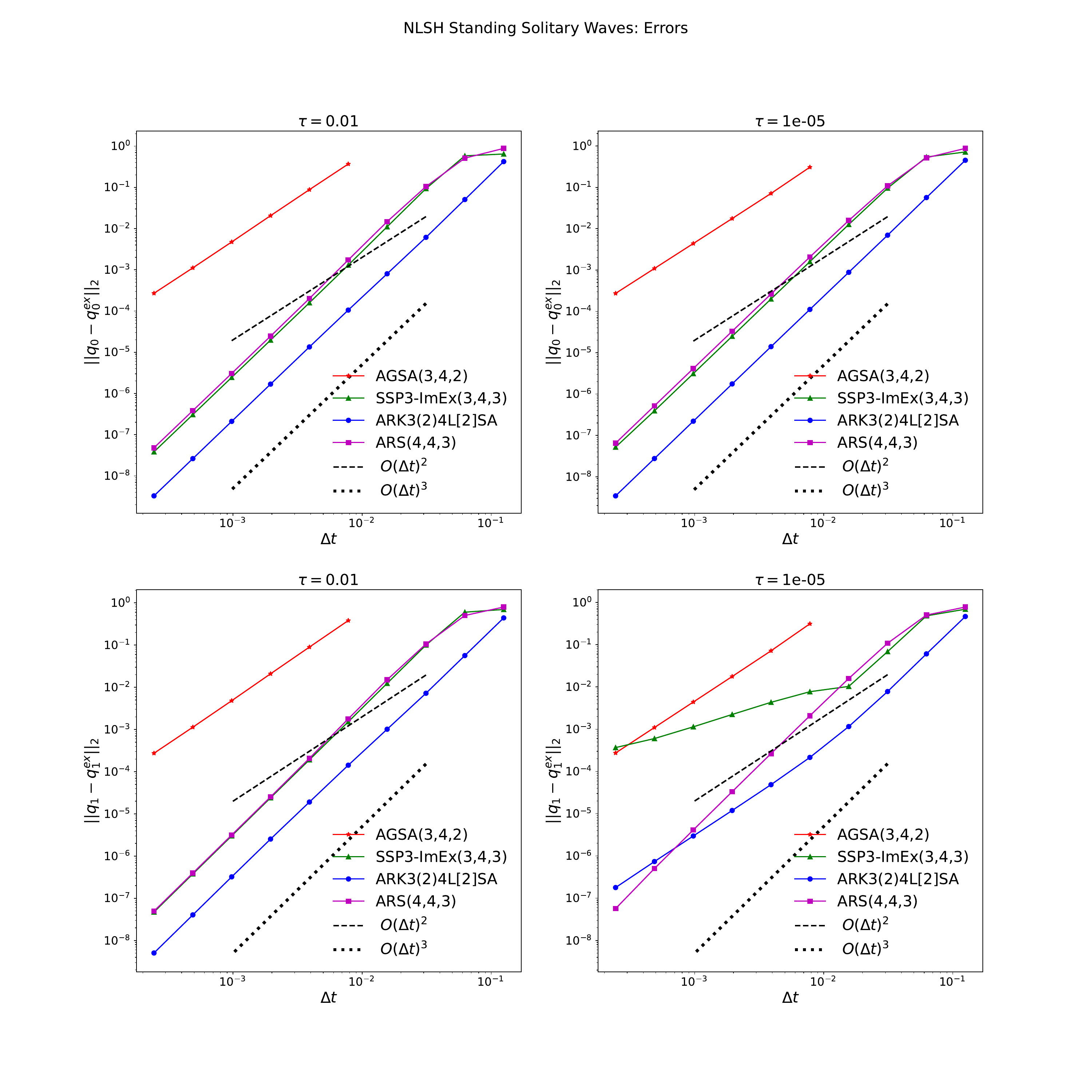}
    \caption{Error convergence for $q_0$ (top row) and $q_1$ (bottom row) for two relaxation parameters. The reference solutions $q_0^{ex}$ and $q_1^{ex}$ 
    are exact solitary (standing) wave solutions. The methods AGSA(3,4,2), SSP3-ImEx(3,4,3), and 
    ARS(4,4,3) exhibit the expected AA property for all components, while ARK3(2)4L[2]SA shows it for the $q_0$-component, beyond the theoretical guarantee.}
    \label{fig:exact_AA}
\end{figure}

Consistent with the theoretical results derived in 
the previous sections, we find that the type I GSA method AGSA(3,4,2) and the type II method ARS(4,4,3) show AA for both components. 
For SSP3-ImEx(3,4,3), the theory guarantees AA for $q_0$ but not for $q_1$, which is exactly what we observe.  Meanwhile, for ARK3(2)4L[2]SA even 
though the theory does not guarantee AA for either component, we do observe AA for $q_0$.

\subsection{RK relaxation and error growth over time}
\label{subsec:relax}
As discussed above, both NLS and NLSH possess conserved quantities, and
numerical conservation of discrete versions of these is highly desirable.
Relaxation techniques \cite{ketcheson2025approximation} offer a minimally
invasive way to modify RK schemes to improve their conservation properties.
Here we apply relaxation to enforce conservation of the the modified mass
($\bar{I}_1$) for NLSH.  Results for the ARS(4,4,3) scheme are shown in Figure
\ref{fig:relax4x4}.
Here we compute:
\begin{itemize}
    \item For discretizations of NLS: the norm of the difference between the
    exact and numerical solutions of the NLS equation.
    \item For discretizations of NLSH: the norm of the difference between the
    exact solution of NLS and the numerical solution of NLSH.
\end{itemize}
Note that in both cases, the ``error" is computed with respect to the exact solution
of the NLS equation; it therefore includes both the numerical error and the
hyperbolization error.  In Figure \ref{fig:relax4x4} we see that:
\begin{itemize}
    \item For large values of $\tau$ and at early times, the hyperbolization error dominates.
    \item The use of relaxation leads to much smaller numerical errors at late times.
\end{itemize}
Because of these two effects, we see for instance that with $\tau=10^{-3}$ or $\tau=10^{-4}$,
the hyperbolization error dominates at early times while the numerical error dominates
later.  Because of this, the NLS solutions have smaller error early on, but eventually
the relaxation solutions have smaller error, and in particular the solution of
NLSH with relaxation is a more accurate approximation of the solution of NLS than what
is obtained by discretizing NLS without relaxation.
Very similar behavior was observed using other time discretizations.

\begin{figure}
    \includegraphics[width=1.2\linewidth]{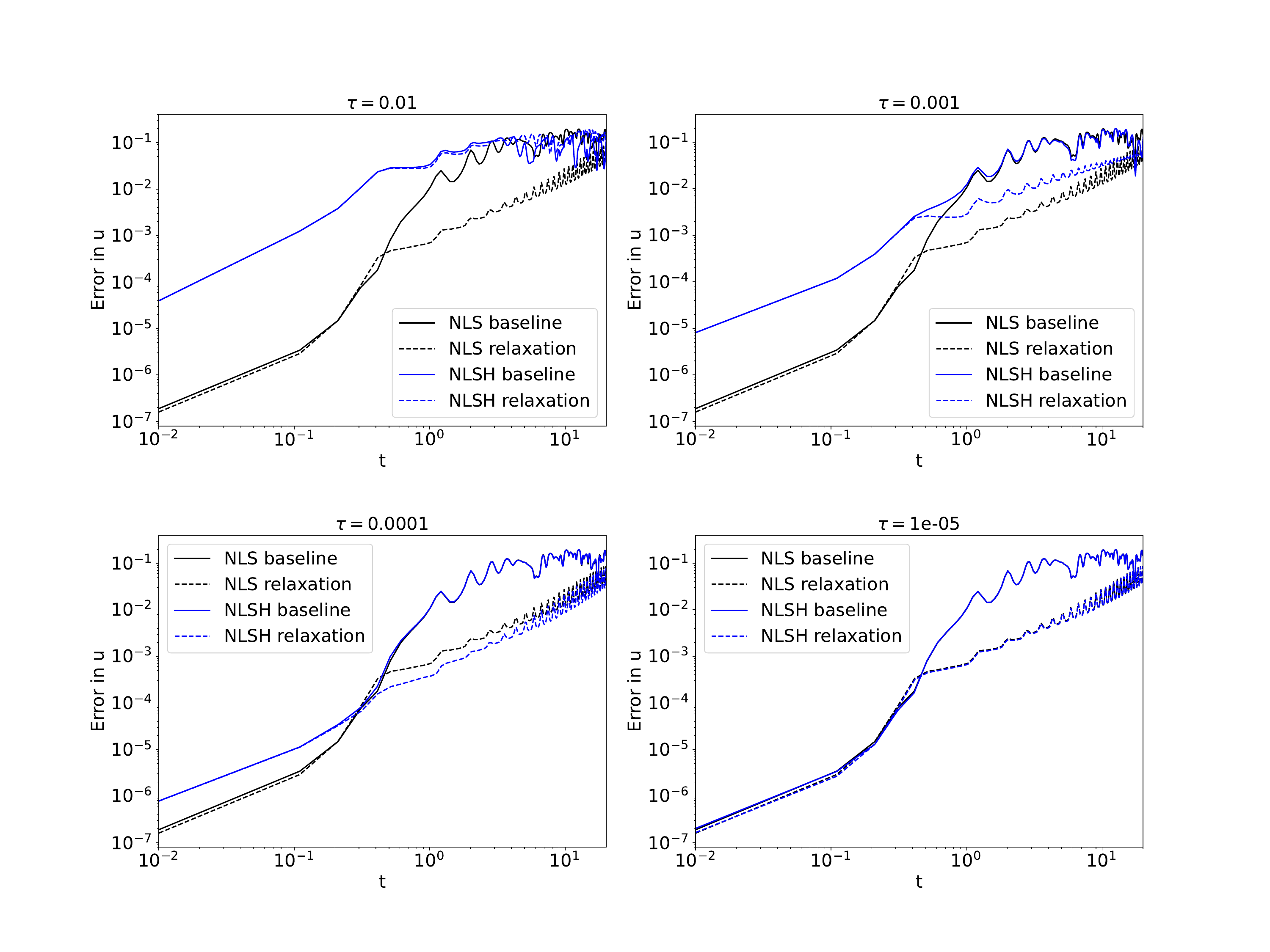}
    \caption{Error versus time for numerical solutions of NLS and NLSH, with
            and without relaxation, for 4 different values of $\tau$. 
            } \label{fig:relax4x4}
\end{figure}


\subsection{A defocusing NLS test case: smoothed Riemann problem}
\label{subsec:DSW}
For $\kappa<0$, \eqref{Eq:NLS} is referred to as the defocusing NLS equation.
In this section we study the approximation of a defocusing NLS solution via the
NLSH system.
We consider a Riemann problem involving the formation of a rarefaction and a
dispersive shock wave.  Following \cite{dhaouadi2019extended}, we work with
hydrodynamic variables $(\rho, \phi)$, which are related to the NLS
solution $u$ as follows:
\begin{align}
    u = \sqrt{\rho}e^{i\theta} \\
    \phi = \nabla\theta
\end{align}

We take $\kappa=-1$
and an initial condition given by a smoothed step function:
\begin{subequations} \label{riemann-IC}
\begin{align}
    \rho(x,t=0) & = \frac{\rho_L+\rho_R}{2} + \left( \frac{\rho_R-\rho_L}{2} \right)\tanh\left(100x\right) \\
    \phi(x,t=0) & = 0.
\end{align}
\end{subequations}
We take $\rho_L = 2$, $\rho_R = 1$, and consider three different values of
$\tau$.  A significant issue arises in studying a Riemann problem when our
numerical method requires the use of periodic boundary conditions; the discontinuity
between the states at the left and right domain boundary constitutes an additional
Riemann problem that generates waves that can overlap with the part of the
solution in which we are interested.  In order to avoid this, we use a very large
domain, $x\in(-1600,1600)$, and end the simulation before significant perturbations
from the boundary Riemann problem have entered the domain of interest.
The whole domain is discretized with $m=16384$ points and $dt=10^{-4}$.

We present the solution in Figure \ref{fig:DSW}, where we plot $\rho$ and
$\phi$ as a function of $x$ at $t=70$.  The plot shows the solution over
the domain $x\in[-335,335]$.
For values of $\tau$ shown, there is relatively good agreement between
the NLSH and NLS solutions.  It should be kept in mind that significantly 
larger values of $\tau$ (for example, $\tau=0.05$) lead to solutions that
differ dramatically from the NLS solution.

\begin{figure}
    \centering
    \includegraphics[width=1.0\linewidth]{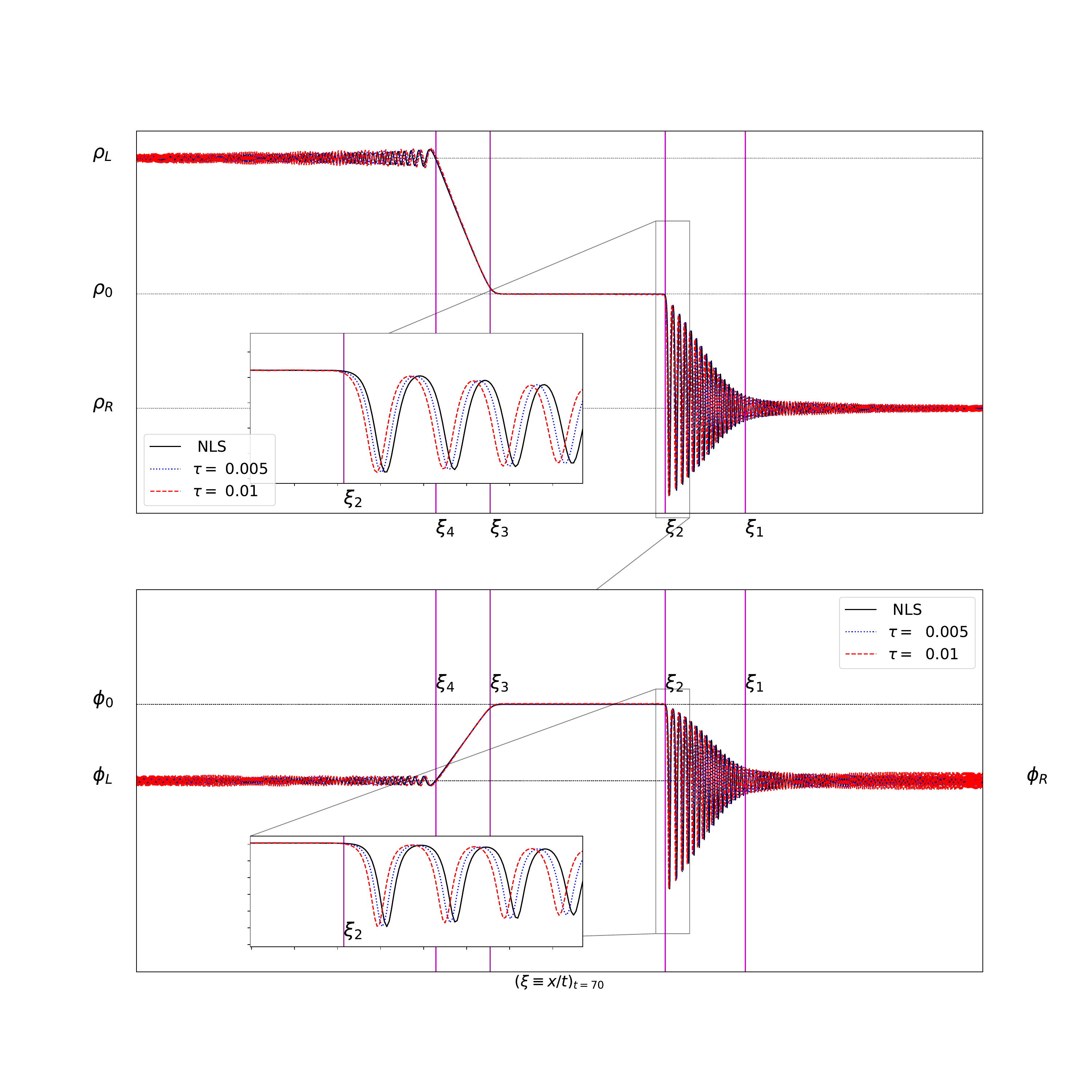}
    \caption{Solution of defocusing NLS at $t=70$ with initial smoothed step function \eqref{riemann-IC}.
    The theoretical edges of the rarefaction and dispersive shock wave are marked by $(\xi_3, \xi_4)$ and
    $(\xi_1, \xi_2)$, respectively.}
    \label{fig:DSW}
\end{figure}

Whitham modulation theory provides estimates of the leading and trailing edges
of the shock and rarefaction waves, which are marked (in terms of the 
self-similarity variable $\xi=x/t$) on the plot and are in
agreement with the numerical results.
For additional details regarding this problem and its solution via modulation
theory we refer the reader to \cite{dhaouadi2019extended}.
It is interesting that both modulation theory and hyperbolic relaxation provide an
approximation of the NLS solution via a hyperbolic PDE system.

\section{Discussion}
In the development of hyperbolic approximations to high-order PDEs, a key
objective is the preservation, as far as possible, of the mathematical
properties of the original model.
As we have demonstrated, the NLSH system inherits several key properties
of the NLS equation: Hamiltonian structure, conserved quantities, and solitary
waves.  In particular, the preservation of three nonlinear conserved quantities
is remarkable; to our knowledge, no other existing hyperbolization preserves more than one.
Compared to the other existing hyperbolic approximation of NLS \cite{dhaouadi2019extended}, the
one presented here is advantageous since it applies in both the focusing and
defocusing regimes, preserves more of the structure of the original equation,
and results in a smaller system of equations.
The energy-preserving and asymptotic-preserving discretizations we have proposed
carry these properties over into the discrete approximations.

In one space dimension, numerical solutions of dispersive wave equations like
NLS can be computed rather efficiently with existing numerical methods.  However,
in higher dimensions and with non-periodic boundary conditions, the numerical solution
of such problems becomes more costly, and we anticipate that extensions of the
hyperbolic approximation developed here may be of significant practical interest.
This is the subject of current work.

At the same time, the fact that the dynamics of an integrable, dispersive nonlinear
wave equation like NLS can be approximated arbitrarily well by a hyperbolic system
is interesting in itself.  There are a wide range of important questions regarding
the NLSH system which have yet to be considered, such as the application of Whitham
modulation theory, the study of periodic waves and dispersive shocks, breathers, and
so forth.  These topics are left for future work.

The preservation of not just one but (at least) three conserved quantities
related to those of the NLS equation is remarkable and seems to be unique among
hyperbolic approximations of higher-order equations.  Even the hyperbolic approximation
of the (integrable) Korteweg-de Vries equation (see \cite{besse2022perfectly,biswas2024kdvh}) seems to possess only one conserved
nonlinear functional.  The question of whether there exist additional conserved
quantities for NLSH remains open.

\bibliographystyle{plain}
\bibliography{refs}

\begin{thebibliography}{10}

\bibitem{ascher1997implicit}
Uri~M Ascher, Steven~J Ruuth, and Raymond~J Spiteri.
\newblock Implicit-explicit {R}unge-{K}utta methods for time-dependent partial
  differential equations.
\newblock {\em Applied Numerical Mathematics}, 25(2):151--167, 1997.

\bibitem{besse2022perfectly}
Christophe Besse, Sergey Gavrilyuk, Maria Kazakova, and Pascal Noble.
\newblock Perfectly matched layers methods for mixed hyperbolic--dispersive
  equations.
\newblock {\em Water Waves}, 4(3):313--343, 2022.

\bibitem{biswas2023multiple}
Abhijit Biswas and David~I Ketcheson.
\newblock Multiple-relaxation {R}unge-{K}utta methods for conservative
  dynamical systems.
\newblock {\em Journal of Scientific Computing}, 97(1):4, 2023.

\bibitem{biswas2024accurate}
Abhijit Biswas and David~I Ketcheson.
\newblock Accurate solution of the nonlinear {S}chr{\"o}dinger equation via
  conservative multiple-relaxation {I}m{E}x methods.
\newblock {\em SIAM Journal on Scientific Computing}, 46(6):A3827--A3848, 2024.

\bibitem{biswas2024kdvh}
Abhijit Biswas, David~I. Ketcheson, Hendrik Ranocha, and Jochen Sch\"utz.
\newblock Traveling-wave solutions and structure-preserving numerical methods
  for a hyperbolic approximation of the {K}orteweg-de {V}ries equation.
\newblock {\em Journal of Scientific Computing}, 103(3):90, 2025.

\bibitem{boscarino2024implicit}
Sebastiano Boscarino, Lorenzo Pareschi, and Giovanni Russo.
\newblock {\em Implicit-Explicit Methods for Evolutionary Partial Differential
  Equations}, volume~24 of {\em Mathematical Modeling and Computation}.
\newblock SIAM, 2024.

\bibitem{boscarino2024}
Sebastiano Boscarino and Giovanni Russo.
\newblock Asymptotic preserving methods for quasilinear hyperbolic systems with
  stiff relaxation: a review.
\newblock {\em SeMA}, 81(3):3--49, 2024.

\bibitem{chesnokov2019hyperbolic}
Alexander Chesnokov and Trieu~Hai Nguyen.
\newblock Hyperbolic model for free surface shallow water flows with effects of
  dispersion, vorticity and topography.
\newblock {\em Computers \& Fluids}, 189:13--23, 2019.

\bibitem{dhaouadi2019extended}
Firas Dhaouadi, Nicolas Favrie, and Sergey Gavrilyuk.
\newblock Extended {L}agrangian approach for the defocusing nonlinear
  {S}chr{\"o}dinger equation.
\newblock {\em Studies in Applied Mathematics}, 142(3):336--358, 2019.

\bibitem{duran2000numerical}
Angel Dur{\'a}n and Jesus~Maria Sanz-Serna.
\newblock The numerical integration of relative equilibrium solutions. the
  nonlinear {S}chr{\"o}dinger equation.
\newblock {\em IMA journal of numerical analysis}, 20(2), 2000.

\bibitem{gavrilyuk2022hyperbolic}
Sergey Gavrilyuk and Keh-Ming Shyue.
\newblock Hyperbolic approximation of the {BBM} equation.
\newblock {\em Nonlinearity}, 35(3):1447, 2022.

\bibitem{shijin2001AP}
Shi Jin and Lorenzo Pareschi.
\newblock Asymptotic-preserving ({AP}) schemes for multiscale kinetic
  equations: A unified approach.
\newblock In Heinrich Freistühler and Gerald Warnecke, editors, {\em
  Hyperbolic Problems: Theory, Numerics, Applications}, volume 141 of {\em ISNM
  International Series of Numerical Mathematics}, page 573–582. Birkhäuser
  Basel, 2001.

\bibitem{kennedy2003additive}
C.~A. Kennedy and M.~H. Carpenter.
\newblock Additive {R}unge-{K}utta schemes for convection-diffusion-reaction
  equations.
\newblock {\em Applied Numerical Mathematics}, 44:139--181, 2003.

\bibitem{ketcheson2019relaxation}
David~I. Ketcheson.
\newblock Relaxation {R}unge–{K}utta methods: Conservation and stability for
  inner-product norms.
\newblock {\em SIAM Journal on Numerical Analysis}, 57(6):2850--2870, 2019.

\bibitem{ketcheson2025approximation}
David~I Ketcheson and Abhijit Biswas.
\newblock Approximation of arbitrarily high-order {PDE}s by first-order
  hyperbolic relaxation.
\newblock {\em Nonlinearity}, 38(5):055002, 2025.

\bibitem{lanczos1949}
Cornelius Lanczos.
\newblock {\em The variational principles of mechanics}.
\newblock Univ. of Toronto Press, 1949.

\bibitem{mazaheri2016first}
Alireza Mazaheri, Mario Ricchiuto, and Hiroaki Nishikawa.
\newblock A first-order hyperbolic system approach for dispersion.
\newblock {\em J. Comput. Phys.}, 321(Supplement C):593--605, 2016.

\bibitem{ranocha2020general}
Hendrik Ranocha, Lajos Lóczi, and David~I. Ketcheson.
\newblock General relaxation methods for initial-value problems with
  application to multistep schemes.
\newblock {\em Numerische Mathematik}, 146:875--906, October 2020.

\bibitem{ranocha2020relaxation}
Hendrik Ranocha, Mohammed Sayyari, Lisandro Dalcin, Matteo Parsani, and
  David~I. Ketcheson.
\newblock Relaxation {R}unge-{K}utta methods: Fully-discrete explicit
  entropy-stable schemes for the compressible {E}uler and {N}avier-{S}tokes
  equations.
\newblock {\em SIAM Journal on Scientific Computing}, 42(2):A612--A638, March
  2020.

\bibitem{toro2014advection}
Eleuterio~F Toro and Gino~I Montecinos.
\newblock Advection-diffusion-reaction equations: hyperbolization and
  high-order {ADER} discretizations.
\newblock {\em SIAM Journal on Scientific Computing}, 36(5):A2423--A2457, 2014.

\bibitem{yang2010nonlinear}
Jianke Yang.
\newblock {\em Nonlinear waves in integrable and nonintegrable systems}.
\newblock SIAM, 2010.

\end{thebibliography}

\end{document}